\documentclass[12pt]{article}

\usepackage{url}

 \usepackage{hyperref}

\usepackage{amsmath}
\usepackage{amssymb}
\usepackage{amsfonts}
\usepackage{graphicx}
\usepackage{tikz}

\usepackage{amsthm}

\newtheorem{prop}{Proposition}[section]
\newtheorem{theo}[prop]{Theorem}

\newtheorem{lemm}[prop]{Lemma}
\newtheorem{remark}[prop]{Remark}
\newtheorem{definition}[prop]{Definition}

\theoremstyle{definition}

\newcommand{\eqco}{\setcounter{equation}{0}}
\newcommand{\allco}{\eqco}
\newcommand{\lbl}{\label}

\newcommand{\Po}{{\cal P}}

\newcommand{\cQ}{{\cal Q}}
\newcommand{\cR}{{\cal R}}
\newcommand{\Q}{{\end{document}}}

\newcommand{\Gum}{{\mathsf{Gu}}}
\renewcommand{\d}{d}

\newcommand{\N}{\mathbb{N}}
\newcommand{\Z}{\mathbb{Z}}

\newcommand{\bH}{\mathbb{H}}

\newcommand{\R}{\mathbb{R}}
\newcommand{\E}{\mathbb{E}}
\newcommand{\EE}{\mathbb{E}}
\newcommand{\X}{\mathcal{X}}
\newcommand{\cX}{\mathcal{X}}

\renewcommand{\Pr}{\mathbb{P}}
\newcommand{\PP}{\mathbb{P}}

\newcommand{\ee}{e}

\renewcommand{\E}{\mathbb E \,}

\newcommand{\tod}{\stackrel{{\cal D}}{\longrightarrow}}

\newcommand{\Y}{{\cal Y}}

\newcommand{\cF}{{\cal F}}

\newcommand{\eps}{\varepsilon}
\newcommand{\edm}{\end{displaymath}}
\def\benu{\begin{enumerate}}
\def\eenu{\end{enumerate}}
\def\beqn{\begin{equation}}
\def\eeqn{\end{equation}}
\def\bea{\begin{eqnarray}}
\def\eea{\end{eqnarray}}
\newcommand{\bean}{\begin{eqnarray*}}
\newcommand{\eean}{\end{eqnarray*}}
\newcommand{\bear}{\begin{eqnarray}}
\newcommand{\eear}{\end{eqnarray}}
\newcommand{\1}{{\bf 1}}
\renewcommand{\epsilon}{\varepsilon}

\newcommand{\bbS}{\mathbb{S}}

\renewcommand{\P}{{\mathbb P}}

\newcommand{\dist}{\,{\rm dist}}

\def\qed{\hfill\hbox{${\vcenter{\vbox{
    \hrule height 0.4pt\hbox{\vrule width 0.4pt height 6pt
    \kern5pt\vrule width 0.4pt}\hrule height 0.4pt}}}$}}

\begin{document}
\title{\bf Covering one point process with another}
%\title{\bf Coverage thresholds, extremal spacings and $U$-records}

\author{
	Frankie Higgs$^{1,4}$,
	Mathew D. Penrose$^{2,4}$ and Xiaochuan Yang$^{3,4}$  \\
%{\normalsize{\em University of Bath}}
}

 \footnotetext{ $~^1$ Department of
Mathematical Sciences, University of Bath, Bath BA2 7AY, United
Kingdom. {\texttt fh350@bath.ac.uk} }

\footnotetext{ $~^2$ Department of
Mathematical Sciences, University of Bath, Bath BA2 7AY, United
Kingdom. {\texttt m.d.penrose@bath.ac.uk} }

\footnotetext{ $~^3$
	     Department of Mathematics, Brunel University London, Uxbridge, UB83PH, United Kingdom.
             \texttt{xiaochuan.yang@brunel.ac.uk
            ORCID:0000-0003-2435-4615}}
             
	     %\and Xiaochuan  Yang \thanks{Department of Mathematics, Brunel University London, Uxbridge, UB83PH, United Kingdom.
             %\texttt{xiaochuan.yang@brunel.ac.uk
            %ORCID:0000-0003-2435-4615}}

 \footnotetext{ $~^4$ Supported by EPSRC grant EP/T028653/1 }

 \footnotetext{
	 MSC: 
	 %AMS classifications:
	 60D05, 
% (Geometric prob and stoch geom)
 %60G55 (Point processes) 
 60F05, %(Central limit and other weak theorems) 
 60F15.} %(strong limit theorems)
 %53A05 (surfaces in Euclidean and related spaces)
 %60G57 (random measures - maybe not)


%Running title: {\bf } \\

%\footnotetext{ AMS classifications: 60K35, 60G55, 82B43 }

%\footnotetext{ Keywords: semi-homogeneous random digraph,
% giant component,branching process}

%\date{}
\maketitle

%\newpage
\begin{abstract}
	Let $X_1,X_2, \ldots $ and $Y_1, Y_2, \ldots$
	be i.i.d. random uniform points in a bounded domain 
	$A \subset \mathbb{R}^2$ with smooth or 
	polygonal boundary.
 Given $n,m,k \in \N$,
	define the {\em two-sample $k$-coverage threshold} $R_{n,m,k}$ to
	be the smallest $r$ such that each point of $ \{Y_1,\ldots,Y_m\}$
	 is covered at least $k$ times 
	 by the disks of radius $r$ centred on $X_1,\ldots,X_n$.
	 We obtain the limiting distribution of $R_{n,m,k}$ as
	 $n \to \infty$ with $m= m(n) \sim \tau n$ for
	 some constant $\tau >0$, with $k $ fixed.
	 If $A$ has unit area,
	 then $n \pi R_{n,m(n),1}^2 - \log n$ is asymptotically Gumbel
	 distributed with scale parameter $1$ and location
	 parameter $\log \tau$. For $k >2$, we find  
	  that $n \pi R_{n,m(n),k}^2 - \log n - (2k-3) \log \log n$
	 is asymptotically Gumbel with  scale parameter $2$
	 and a more complicated location parameter involving the perimeter
	 of $A$; boundary effects dominate when
	 $k >2$.  For $k=2$ the limiting cdf is a two-component
	 extreme value distribution
	 %the product of two Gumbel cdfs 
	 with scale  parameters 1 and 2.
	 We also give analogous results for higher dimensions,
	 where the boundary effects dominate for all $k$.

	% and also a strong law of large numbers for $R_n$ in the large-$n$ limit.  For example, if $d=3$ and $A$ has volume 1 and perimeter $|\partial A|$ then $\Pr[n\pi R_n^3 - \log n - 2 \log (\log n) \leq x]$ converges to $\exp(-2^{-4}\pi^{5/3} |\partial A| e^{-2 x/3})$, and $(n \pi R_n^3)/(\log n) \to 1$ almost surely.

%	We give similar results for general $d$, and also for the case where $A$ is a polytope.  We also generalize to allow for multiple coverage.  The analysis relies on classical results by Hall and by Janson, along with a careful treatment of boundary effects. For the strong laws of large numbers, we can relax the requirement that the underlying density on $A$ be uniform.
\end{abstract}

%\section{Statement of results}
%\label{StateResult}

\section{Introduction}
\label{SecIntro}

This paper is primarily concerned with the following 
{\em two-sample random coverage} problem. Given a 
specified compact region $B$ in a $d$-dimensional Euclidean space, 
suppose $m$ points $Y_j$ are placed randomly in $B$.
What is the probability that these $m$ points are fully covered by a union
of Euclidean  balls of radius $r$ centred on $n$ points $X_i$ placed
independently uniformly at random in $B$, in the large-$n$ limit
with $m= m(n)$ becoming large and
 $r =r(n)$ becoming  small in an appropriate manner?

 % This is a very natural type of question with
% a long history; see for example
%\cite{BB,CSKM,Flatto,HallZW,HallBk,Janson,Moran}. Potential applications include
%wireless communications \cite{BB,Lan}, ballistics \cite{HallBk}, genomics
%\cite{Roy}, statistics \cite{Cuevas},  immunology \cite{Moran},
%and topological data analysis \cite{BW,KTV}.

In an  alternative version of this question,
 the $X$-points are placed uniformly not
in $B$, but in a larger region $A$ with $B \subset A^o$
($A^o$ denotes the interior of $A$).
This version is simpler because boundary effects are avoided.
We consider this version too.

%As more points get added, we can cover $B$
%using smaller balls. We consider here the
We shall express our results in terms of the
 {\em two-sample coverage threshold} $R_{n,m}$, which we
define to be the smallest radius of balls, centred on a
 set $\X_n$ of $n$ independent uniform random points in $A$,
 required to cover all the points of a sample
 $\Y_m$ of $m$ uniform random points in $B$.
More generally, for $k \in \N$
 the {\em two-sample $k$-coverage threshold} $R_{n,m,k}$
is the smallest radius required to cover $\Y_m$ $k$ times. 
These thresholds are random variables, because the locations of
the centres are random.
We investigate their probabilistic behaviour 
as $n$ and $m$ become large.

A related  question is to ask for
coverage of the whole set $B$, not just of the
point set $\Y_m$. We refer here to the smallest
radius $r$ such that $B$ is contained in the union
of the balls of radius $r$ centred on points of $A$,
as the {\em complete coverage threshold}.
The asymptotic behaviour of this threshold
has been
addressed
in
%by Hall 
\cite{HallZW} and 
%Janson 
\cite{Janson} (for the case with $B \subset A^o$)
and in \cite{Pen22} (for the case with $B=A$).
Clearly $R_{n,m}$ provides a lower bound for the 
complete coverage threshold.

Also related is the problem, when $m=n$ and $B=A$, of 
finding the {\em matching threshold}, that is, the
minimum $r$ such that a perfect bipartite matching of the samples
$\X_n$ and $\Y_n$ exists with all edges of length at most $r$.
%
%providing a complete
%matching of the two samples in a manner that minimises the
%maximum Euclidean edge length of the matching. 
This
problem has been considered in \cite{Leighton,ShorY}, with
applications to the theory of empirical measures.
See e.g. \cite{GTS} for recent application of results
in \cite{Leighton,ShorY} to clustering and classification problems
in machine-learning algorithms.

Our problem
is different since we allow the 
%$ polygamous matching of the
$X$-points to practice polygamy, and require all of the $Y$-points, but not
necessarily all of the $X$-points, to be matched.
Clearly $R_{n,n}$ is a lower bound for the matching threshold.
This lower bound is asymptotically of a different order of magnitude than the
matching threshold when $d=2$, 
but the same order of magnitude
when $d \geq 3$. A slightly better lower bound is given by
$\tilde{R}_{n,n}$, which we define to be the smallest
$r$ such that all $Y$-points are covered by $X$-points {\em and}
all $X$-points are covered by $Y$-points.
We  expect that our methods can be used to show that
$\lim_{n \to \infty} \Pr[\tilde{R}_{n,n} \leq r_n]
= \lim_{n \to \infty} \Pr[R_{n,n} \leq r_n]^2$ for any sequence
$(r_n)$ such that the limit exists, but proving this is beyond the
scope of this paper.
It is tempting to conjecture that 
the lower bound $\tilde{R}_{n,n}$
for the matching threshold
might perhaps be asymptotically sharp
as $n \to \infty$ in sufficiently
high dimensions.

Another related problem is that of understanding the
{\em bipartite connectivity threshold.} Given $\X_n, \Y_m$ and
$r>0$,
we can create a {\em bipartite random geometric graph}
(BRGG)
on vertex set $\X_n \cup \Y_m$ by drawing an edge between
any pair of points $x \in \X_n, y \in \Y_m$ a distance
at most $r$ apart. The bipartite  connectivity threshold is the
smallest $r$ such that this graph is connected, and
the two-sample coverage threshold $R_{n,m}$  is a lower bound for
the bipartite connectivity threshold. Two related thresholds
are:
the smallest $r$ such that each point of 
%the $\Y$-sample
 $\Y_m$
is connected  by a path in the BRGG to at least one other 
point of
%the $\Y$-sample,
 $\Y_m$,
and
the smallest $r$ such that any two points of 
%the $\Y$-sample
 $\Y_m$,
are connected  by a path in the BRGG
(but isolated points
in $\X_n$ are allowed in both cases). Provided $m \geq 2$, these thresholds
both lie
between $R_{n,m}$ and the bipartite connectivity threshold,
and have been studied in \cite{IY,PAB}.

Motivation for considering coverage problems comes from
wireless communications technology
(among other things); one may be interested in
covering  a region of land by
mobile wireless transmitters (with locations modelled as
the set of random points $X_i$). If interested in covering
the whole region of land, one needs to consider the
complete coverage threshold.
In practice, however,
it may be sufficient to cover not the whole region
but a finite collection of receivers placed in that region (with locations
modelled as the set of random points $Y_j$), and the
two-sample coverage threshold addresses this problem.
See \cite{IY} for further discussion of motivation from
wireless communications.

See also \cite{BI13}, which discusses a similar model where
the $\Y$-sample represents
a set of `sensors' which cover space over short distances, and 
the $\cX$-sample represents a set of `backbone nodes' which communicate
over longer distances.
In \cite{BI13}
the interest is in the volume of the region of space that is covered 
by sensors that are themselves covered by backbone nodes; a central
limit theorem is derived for the volume of the complementary 
region. The quantity of interest to us  here corresponds to
the probability that
all of the sensors are covered (at least $k$ times) by backbone nodes.

%	.  Another way to avoid boundary effects would be
%to consider coverage of a smooth manifold such as a sphere
%(as in \cite{Moran}), and this
%was also addressed in \cite{Janson}.

We shall determine the limiting behaviour of $\Pr[R_{n,m(n),k} \leq r_n]$
for any fixed $k$, any 
sequence $m(n)_{n \geq 1} $ of integers asymptotically proportional to $n$,
and
any sequence  of numbers $(r_n)$ such that the limit exists,
for the case where $B$ is smoothly bounded (for general $d \geq 2$)
or where
$B$ is a polygon (for $d =2$). 
%For $d=2$ we consider general $k$; 
%for $ d \geq 3$ we consider only the case with $k=1$ (perhaps).
We also obtain similar  results for  the Poissonized versions of this problem.

Our results show that when $d \geq 3$ the boundary effects dominate,
i.e. the point of the $\Y$-sample furthest from its $k$-nearest neighbour
in the $\X$-sample is likely to be near the boundary of $B$.
%(we show
%this for $k =1$ but we believe it is also true for all $k$ when
%$d\geq 3$ -{\bf see Section \ref{s:dkhi}}).
When $d=2$, boundary effects are negligible for $k=1$  but
dominate for $k \geq 3$. When $d=k=2$ the boundary and interior 
effects are of comparable importance; the point of the $\Y$-sample
furthest from its second-nearest neighbour in the $\X$-sample
has non-vanishing probability of being near the boundary of $B$ but also
 non-vanishing probability of being in the interior.

%[ Further stuff we could add - could do the above for larger $d$.
%We might also derive strong laws of large numbers
%showing that that $nR_{n,m(n),k}^d/\log n$
%converges  almost surely to a finite positive limit,
%and establishing the value of the limit. These 
%strong laws should carry over
%to  more general cases where $k$ may vary with $n$, and  the distribution
%of points may be non-uniform. This is perhaps more interesting
%if we really believe the lower bound for the matching threshold
%might be sharp at the SLLN level.

%We restrict attention here to coverage by Euclidean balls of equal radius.
%We could consider generalizations such as
%other shapes or variable radii. ]

In Section~\ref{sim-section} we discuss the results of computer experiments, in which we sampled many independent copies of $R_{n,m(n),k}$ and plotted the estimated distributions of these radii (suitably transformed so that a weak law holds) alongside the limiting distributions we state in Section~\ref{secweak}.
These experiments motivated a refinement to our limit results, in which we explicitly included the leading-order error term, so that we can approximate the distribution of $R_{n,m(n),k}$ well for given finite $n$.

We work within the following mathematical framework.  Let $d \in \N$.
 Let $A \subset \R^d$ be compact.
Let $B \subset A$ be a specified Borel set (possibly
the set $A$ itself) with a nice boundary
(in a sense to be made precise later on), and with
volume $|B|>0$.
Suppose on some probability space $(\bbS,\cF,\Pr)$ that
$X_1,Y_1,X_2,Y_2,\ldots$ are independent random $d$-vectors 
with $X_i$ uniformly distributed over $A$ and $Y_i$
uniformly distributed over $B$ for each $i \in \N$.
%identically distributed
%random $d$-vectors with common probability distribution
%$\mu$ having density function
%$f$ with support $A$, and $Y_1,Y_2,\ldots$ are
%independent of .
For $x \in \R^d$ and $r>0$ set $B_r(x) :=
B(x,r):= \{y \in \R^d:\|y-x\| \leq r\}$
where $\|\cdot\|$ denotes the Euclidean norm.
For $n \in \N$,  let $\X_n:= \{X_1,\ldots,X_n\}$ and
let $\Y_{n,B}:= \{Y_1,\ldots,Y_n\}$.
Given also $m, k \in \N$, we
 define the $k$-coverage threshold $R_{n,m,k}$  by
\bea
R_{n,m,k}(B) : =
 \inf \left\{ r >0: \X_n   (B(y,r)) \geq k 
~~~~ \forall y \in \Y_{m,B} \right\},
~~~ n,m,k  \in \N,
\label{Rnkdef}
\eea 
where for any point set $\X \subset \R^d$ and any $D \subset \R^d$ we write
$\X(D)$ for the number of points of $\X$ in $D$,
 and we use the convention $\inf\{\} := +\infty$.
In particular $R_{n,m}(B) : = R_{n,m,1}(B)$ is the two-sample
coverage threshold.  Observe that
$R_{n,m}(B)  = \inf \{ r >0: \Y_{m,B} \subset \cup_{i=1}^n B(X_i,r) \}$.

We are mainly interested in the case with
$B=A$. In this case
we write simply $\Y_m$, $R_{n,m,k}$
and $R_{n,m}$ for $\Y_{m,A}$, $R_{n,m,k}(A)$ and $R_{n,m}(A)$ respectively.

We are interested in the asymptotic behaviour of $R_{n,m}(B)$ 
 for large $n,m$; in fact we take $m$ to be asymptotically proportional
 to $n$.  More generally, we consider $R_{n,m,k}(B)$ for
 fixed $k \in \N$.
 %where $k$ may vary with $n$ (or may be fixed). 

 We also consider analogous 
quantities denoted
 $R'_{t,u}(B)$ and $R'_{t,u,k}(B)$ respectively,
defined similarly using  Poisson samples of points. To define these formally,
let  $(Z_t,t\geq 0)$ be a unit rate Poisson counting
process, independent of $(X_1,Y_1,X_2,Y_2,\ldots)$
and on the same probability space $(\bbS,\cF,\Pr)$ 
(so $Z_t$ is Poisson distributed with mean $t$ for each $t >0$).
Let  $(Z'_t,t\geq 0)$ be a second unit rate Poisson counting
process, independent of $(X_1,Y_1,X_2,Y_2,\ldots)$ and of $(Z_t,t \geq 0)$.
The point process $\Po_t:= \{X_1,\ldots,X_{Z_t}\}$  is 
a Poisson point process in $\R^d $ with intensity measure $t \mu$,
where we set $\mu$ to be the uniform  distribution over  $A$ 
 (see e.g. \cite[Proposition 3.5]{LP}).
The point process $\cQ_{t,B} := \{Y_1,\ldots,Y_{Z'_t}\}$  is 
a Poisson point process in $\R^d $ with intensity measure $t \nu$,
where we set $\nu$ to be the uniform  distribution over  $B$. 
Then 
for $t,u \in (0,\infty), k \in \N$ we define 
%a secondary $k$-coverage threshold
\bea
R'_{t,u,k}(B) : =
%R_{Z_t,Z'_u,k}(B) :=
 \inf \left\{ r >0: \Po_t
(B(y,r)) \geq k 
~ \forall y \in \cQ_{u,B} \right\},
%~~ t,u >0,
\label{Rdashdef}
\eea
with $R'_{t,u}:= R'_{t,u,1}$.
%Also define
%\bea
%\tilde{R}_{Z_t,Z'_u,k}(B) : = \inf \left\{ r >0: 
%\Po_t (B(x,r)) \geq k ~~~ \forall x \in \cQ_{u,B} \cap A^{(r)}
%\right\},~~~ n,k \in \N.
%\label{eqmaxspacPo}
%\eea
When $B=A$ we write simply
%$\tilde{R}_{n,m,k}$, $\tilde{R}_{n,m}$,
%$\cQ_{t}$, $R'_{t,u,k}$, $R'_{t,u,k}$, $\tilde{R}_{Z_t,Z'_u,k}$
$\cQ_{t}$, $R'_{t,u,k}$, $R'_{t,u}$, 
%$\tilde{R}_{Z_t,Z'_u,k}$
for
% $\tilde{R}_{n,m,k}(A)$, $\tilde{R}_{n,m}(A)$,
$\cQ_{t,A}$, $R'_{t,u,k}(A)$, $R'_{t,u,k}(A)$, 
%$\tilde{R}_{Z_t,Z'_u,k}(A)$
respectively.

We mention some notation used throughout. For
$D \subset \R^d$, let $\overline{D}$ 
%and $D^o$
denote the closure of $D$.
%For Borel $D \subset \R^d$,
Let $|D|$ denote the Lebesgue  measure (volume) of $D$, and
$|\partial D|$ the perimeter of $D$, i.e. the
$(d-1)$-dimensional Hausdorff measure of
$\partial D$, when these are defined.
Given $t >1$, we write $\log \log t$ for $\log (\log t)$.
%and interior of $D$, respectively. 
Let $o$ denote
the origin in $\R^d$.
 
%Given two  sets $\X,\Y \subset \R^d$, we
% set  $ \X \triangle \Y := (\X \setminus \Y) 
%\cup (\Y \setminus \X)$, the symmetric difference between $\X$ and $\Y$.
%Also, we write $\X \oplus \Y$ for the set $\{x+y: x \in \X, y \in \Y\}$.
% Given also $x \in \R^d$ we write $x+\Y$ for $\{x\} + \Y$.

%Given $x,y \in \R^d$, we denote by $[x,y]$ the line segment from
%$x$ to $y$, that  is, the convex hull of the set $\{x,y\}$.
%We write $a \wedge b$ (respectively $a \vee b$) for the minimum
%(resp. maximum) of any two numbers $a,b \in \R$.

Let $\theta_d$ denote the volume of a unit radius ball in $\R^d$.
Set $f_0 := 1/|A|$.

%Throughout this paper, $c$ and $c'$ denote positive finite
%constants whose values may vary from line to line and do not depend
%on $n$.
If $t_0 \in (0,\infty)$ and
$f(t),g(t)$ are two functions, defined for all $t \geq t_0$
with $g(t) >0$ for all $t \geq t_0$,
the notation $f(t)= O(g(t))$ as $t \to \infty$
means that $\limsup_{t \to \infty} (|f(t)|/g(t)) < \infty$,
and the notation $f(t)= o(g(t))$ as $t \to \infty$
means that $\limsup_{t \to \infty} (|f(t)|/g(t)) =0.$
If also $f(t) >0$ for all $n \geq n_0$, we use notation
$f(t)= \Theta(g(t))$ to mean that both $f(t)=O(g(t)$
and $g(t)= O(f(t))$.

\section{Statement of results}
\label{secweak}
\allco

Our results are concerned with 
weak convergence  for $R_{n,m,k}(B)$ (defined at (\ref{Rnkdef}))
as $n \to \infty$ with $k$ fixed and $m$ asymptotically proportional
to $n$.
We also give similar results for  $R'_{t,\tau t,k}$,
defined at (\ref{Rdashdef}),
as $t \to \infty$ with $\tau >0$ also fixed.
%and also
%for $\tR_{n,m,k}$ and $\tR_{Z_t,Z'_{\tau t},k}$. 
In all of these limiting results we are taking the variable $n$
to be integer-valued and $t$ to be real-valued.
%in cases where $f$ is uniform on $A$ and $B=A$.

Recall that our $\X$-sample is of points uniformly distributed
over a compact region $A \subset \R^d$, and the $\Y$-sample
is of points in $B$, where
 $B \subset A$ has a `nice' boundary. We now
make this assumption more precise. We always assume one of the
following:  \\

A1: $d \geq 2$ and $B=A$ and $A$ has a $C^{1,1}$ boundary 
and $\overline{A^o}=A$, or

A2: $d = 2$ and $B=A$ and $A$ is polygonal, or 

A3: $d \geq 2$ and $\overline{B} \subset A^o$,
and $B$ is Riemann measurable with $|B| >0$.
% (Recall that Riemann measurability of a compact set
%in $\R^d$ amounts to its boundary having zero Lebesgue measure). 
 (Recall that a compact set $B$ is said to be Riemann measurable  if
 $\partial B $ is Lebesgue-null.) \\

%Our next result concerns the case where $A$ has a smooth boundary and $B=A$.
We say that $A$ {\em has a $C^{1,1}$ boundary} if 
for each $x \in \partial A$
there exists a neighbourhood $U$ of $x$ and a real-valued function $f$ that
is  defined on an open set in $\R^{d-1}$
and Lipschitz-continuously differentiable, such
that  $\partial A \cap U$, after a rotation, is the graph of the
function $f$. 
The $C^{1,1}$ boundary condition is milder than the
$C^2$ boundary condition that was imposed for analogous
results on the complete coverage threshold in \cite{Pen22}.
The extra condition $\overline{A^o} =A$ should also have
been included in  \cite{Pen22} to rule out examples
such as the union of a disk and a circle in $\R^2$.

For compact $A \subset \R^d$ satisfying A1 or A2,
let $|A|$ denote the volume (Lebesgue measure) of $A$ and $|\partial A|$
the perimeter of $A$, i.e. the $(d-1)$-dimensional Hausdorff measure
of $\partial A$, the topological boundary of $A$. Also define
\begin{align}
\sigma_A := \frac{|\partial A|}{|A|^{1-1/d}}.
	\label{e:defsigA} 
\end{align}
Note that $\sigma_A$ is invariant under scaling of $A$,
and is at least $d \theta_d^{1/d} $ by the isoperimetric inequality.
Sometimes
$\sigma_A^d$ is called the {\em isoperimetric ratio}
of $A$.

Our first result concerns
the case with $\overline{B} \subset A^o$. Recall that $f_0:=1/|A|$.
\begin{theo}[Fluctuations of $R_{n,m,k}(B)$ when $\overline{B} \subset
	A^o$]
\label{Hallthm}
Suppose A3 applies.
	Let $k \in \N$ 
	%with $k \geq 2$
	and 
	$\tau >0, \beta \in \R$. Let $m: \N \to \N$, and assume
	$\tau_n: = m(n) /n \to \tau$
	as $n \to \infty$. Then 
	as $n \to \infty$
	we have
	\begin{align}
	%\lim_{n \to \infty} \Pr[ n \theta_d
		& \Pr[ n \theta_d
	f_0  R_{n,m(n),k}(B)^d   - \log n   -
	(k-1) \log \log n \leq \beta] 
		%~~~~~~~~~~~~~~~
		\nonumber \\
		& ~~~~ = 
	\exp	\left(
	 - \frac{\tau_n e^{-\beta}(k-1)^2 \log \log n}{(k-1)!\log n}
		\right)
		e^{-(\tau_ne^{-\beta})/(k-1)!} 
	+ O((\log n)^{-1}). 
		\label{0829a}
		\end{align}
	Also as $t \to \infty$ we have
	\begin{align}
		& \Pr[ t \theta_d f_0 (R'_{t,\tau t,k}(B))^d   
	- \log t   -
	(k-1) \log \log t \leq \beta]
		%~~~~~~~~~~~~~~~
\nonumber \\
		& ~~~~ =
	\exp	\left(
	 - \frac{\tau e^{-\beta}(k-1)^2 \log \log t}{(k-1)!\log t}
		\right)
		e^{-(\tau e^{-\beta})/(k-1)!}
	+ O((\log t)^{-1}). 
	%= \exp(-(\tau e^{-\beta}/(k-1)!) ).
%\nonumber \\
\lbl{0114a}
	\end{align}
\end{theo}
%\begin{remark} {\rm
	{\bf Remarks.}
	1. Given $\xi \in  \R$, $\theta \in (0,\infty)$, let $
	\Gum_{\xi,\theta}$ denote a Gumbel random variable
	with location parameter $\xi$ and scale parameter $\theta$,
	i.e. with cumulative distribution function (cdf)
	$F(x) = \exp(-e^{-(x-\xi)/\theta})$.
	Since the right hand side
	of \eqref{0829a} converges to $\exp(-(\tau e^{-\beta})/(k-1)!)$
	as $n \to \infty$,
it follows from \eqref{0829a} that as $n \to \infty$
	we have the convergence in distribution:
$$
		 n \theta_d
	f_0  R_{n,m(n),k}(B)^d   - \log n   -
	(k-1) \log \log n  \tod \Gum_{\log (\tau/(k-1)!),1}.
$$
Similarly, as $t \to \infty$,
	by \eqref{0114a} we have
$$
		 n \theta_d
	f_0  R'_{t,\tau t,k}(B)^d   - \log t   -
	(k-1) \log \log t  \tod \Gum_{\log (\tau/(k-1)!),1}.
$$
%where $Z$ is a Gumbel random variable with $\Pr[Z \leq x]
%= \exp(- \tau e^{-x}/(k-1)!)$, $x \in \R$. \\
	%} \end{remark}

2.	The $O((\log n)^{-1})$ term in \eqref{0829a}
	and the
	 $O((\log t)^{-1})$ term in \eqref{0114a}
	come partly from an error bound
of $O((\log t)^{1-d})$ in a Poisson approximation
	for the number of isolated points; see Lemma \ref{l:Poapprox}.
	 If $d \geq 3$
the error bound in the Poisson approximation is of higher order,
and hence we can give a more accurate approximation
with an explicit $(\log n)^{-1}$ term (respectively, $(\log t)^{-1}$ term)
included in the first exponential factor on the right,
and an error
of $O((\frac{\log \log n}{\log n})^2)$ in \eqref{0829a} (resp., of
 $O((\frac{\log \log t}{\log t})^2)$ in \eqref{0114a}).
	See \eqref{e:Bdetail} and \eqref{e:BBdet} 
in	the proof of Theorem \ref{Hallthm} for details. \\
%	}
%\end{remark}

%\begin{remark} {\rm

All of our remaining results are for the case $B=A$.

First we briefly discuss the case where
$A$ is the $d$-dimensional unit {\em torus}.
(and $B=A$). In this case, taking $f_0=1$, 
we can obtain exactly the same result
as stated in Theorem \ref{Hallthm},  by the same proof.
We note that a result along these lines 
(for $k=1$ only) 
has been provided previously (with a different proof)
in \cite[Theorem 3.2]{IY},
for $\tau$ large. \cite{IY} is more concerned with the threshold
$r$ such that each vertex of $\Y_m$ has a path to at least one other
point of $\Y_m$ in the BRGG. In any event, the authors of \cite{IY} explicitly
restrict attention to the torus, in their words, 
to `nullify some of the technical
complications arising out of boundary effects'.
In our next results, we embrace these technical complications.

We next give our main result for $d=2$, $k=1$.
%We
%define the constant
%\begin{align}
%c_d:= \theta_{d-1}^{-1} (\theta_d/(2-2/d))^{1-1/d}.
%	\label{e:cdef}
%\end{align}

\begin{theo}[Fluctuations of $R_{n,m}$ in a planar region with boundary]
\label{thsmoothgen}
Suppose $d=2$ and A1 or A2 holds.
	Set $f_0 := |A|^{-1}$.
	Let $\beta,\tau \in \R$ with $\tau >0$. Suppose $m:\N \to \N$
	with $\tau_n := m(n)/n \to \tau$ as $n \to \infty$. Then
	as $n \to \infty$
	\begin{align}
		\Pr \left[ n \pi f_0 R_{n,m(n)}^2
	- \log n \leq \beta \right] 
		=
		\exp \Big(- \frac{\tau_n \pi^{1/2} \sigma_A e^{-\beta /2}  }{
			2 (\log n)^{1/2}} \Big) 
		e^{ -  \tau_n e^{- \beta} }
		+ O((\log n)^{-1}) . 
		\label{0829b}
	\end{align}
	Also, as $t \to \infty$,
\bea
%\lim_{n \to \infty}
%	\Pr \left[ n \pi f_0 R_{n,m(n)}^2
%	- \log n \leq \beta \right] =
%	\lim_{t \to \infty}
	\Pr \left[ t \pi 
	f_0 (R'_{t,\tau t})^2 - \log t \leq \beta \right]
	= 
		\exp \Big(- \frac{\tau \pi^{1/2} \sigma_A e^{-\beta /2} }{
			2 (\log t)^{1/2}} \Big) 
	\exp \left( -  \tau e^{- \beta} \right) 
		+ O((\log t)^{-1}) . 
\label{eqmain3}
\eea
%Also
%\bea
%	\lim_{n \to \infty}[\Pr[n \pi f_0 (R_{n,m(n),k})^2 - \log n
%	- (2k-3) \log \log n \leq \beta] \nonumber \\
%	= \lim_{t \to \infty}[\Pr[t\pi f_0 (R'_{t,\tau t,k})^2 - \log t
%	- (2k-3) \log \log t \leq \beta] \nonumber \\
%%	= \begin{cases}
%		\exp( - \tau  
%		( e^{-\beta} +(1/4) 
%		\pi^{1/2} \sigma_A e^{-\beta/2}) )
%		& ~~ {\rm if} ~ k=2
%		\\
%		\exp( - \tau  (k-1)!^{-1} 2^{-k }
%		\pi^{1/2}   \sigma_A 
%		e^{-\beta/2})
%		& ~~ {\rm if} ~ k >2.
%	\end{cases}
%	\label{eq:khi}
%	\eea
	\end{theo}

	%\begin{remark} {\rm
		{\bf Remark.}
		It follows from \eqref{0829b} that $n \pi f_0 R_{n,m(n)}^2
		-\log n \tod \Gum_{\log \tau,1}$.
%		{\bf [Add remark that 
%		 $n \pi f_0 R_{n,m(n)}^2 - \mu(n \pi f_0 R_{n,m(n)}^2)
%		 \tod \Gum_{\log \log 2,1}$ (Gumbel with scale parameter 1,
%		 centred so median is 0)? Here $\mu(\cdot)$ is median]}
		Denoting the median of the distribution of
		any continuous random variable $Z$ by $\mu(Z)$,
		%$n \pi f_0 R_{n,m(n)}^2$ by $\mu( n \pi f_0 R_{n,m(n)}^2 ),$
		%then
		we have
		$\mu(n\pi f_0 R_{n,m(n)}^2) = \log n +\mu(\Gum_{\log \tau,1}) + o(1)$.
		We can subtract the medians from both sides, and then we have
		$n \pi f_0 R_{n,m(n)}^2 - \mu(n \pi f_0 R_{n,m(n)}^2)
		\tod \Gum_{\log \log 2,1}$,
		where $\Gum_{\log\log 2,1}$ is a Gumbel random variable with scale parameter 1 and median 0.
		The second row of Figure~\ref{simulations1} illustrates
		each of these two convergences in distribution.
		It is clearly visible that subtracting the median gives
		a much smaller discrepancy between the distribution of
		$n \pi f_0 R_{n,m(n)}^2 - \mu(n \pi f_0 R_{n,m(n)}^2)$
		and its limit, suggesting that
		$\mu(n\pi f_0 R_{n,m(n)}^2)-\log n \to \mu(\Gum_{\log \tau,1})$
		quite slowly.
		However, we estimated $\mu(n \pi f_0 R_{n,m(n)}^2)$
		using the sample median of a large number of independent
		copies of $n \pi f_0 R_{n,m(n)}^2$.
		When applying estimates such as \eqref{0829b} to real data,
		a large number of samples may not be available,
		and we do not currently have an expression for 
		$\mu(n\pi f_0 R_{n,m(n)}^2)-\log n - \mu(\Gum_{\log \tau,1})$.
		%where $Z$ is a Gumbel random variable with 
		%cumulative distribution function (cdf)
		%$\Pr[Z \leq x ]=  \exp(-\tau e^{-x})$ for all $x \in \R$.
%

		Simulations with $A$ taken to be a disk or square
		suggest that even for quite large
		values of $n$, with $m(n) = \lfloor \tau n \rfloor$ for
		some fixed $\tau$, the estimated cdf of 
		$n \pi f_0 R_{n,m(n)}^2 - \log n$ from simulations does
		not match the limiting Gumbel cdf particularly well.
		This can be seen in the bottom-left plot of Figure~\ref{simulations1},
		where the estimated cdf (the blue curve)
		is not well-approximated by the limit (the black dashed curve).
		This is because
		the multiplicative  correction factor of
		$\exp(-\tau_n (\pi^{1/2}/2) \sigma_A e^{-\beta/2}
		  (\log n)^{-1/2})$, which we see in
		\eqref{0829b},  tends to 1 very slowly. 
		(We have written it as a multiplicative correction
		to ensure that the right hand side is a genuine
		cdf in $\tau_n$
			plus an $O((\log n)^{-1})$ error term.)

		If instead we  compare the cdf of $
		n \pi f_0 R_{n,m(n)}^2$ estimated by simulations
		with  the corrected cdf $F(x) = 
		\exp(- \frac{\tau_n \sigma_Ae^{-x/2}}{(\log n)^{1/2}})
		\exp(-\tau_n e^{-x})$,
		illustrated as a red dotted line in the same part of Figure~\ref{simulations1},
		we get a much better match. \\
		%} \end{remark}

	Next we give results for $d=2, k \geq 2$ and for
	$d \geq 3$. Given $(d,k)$
%	also $k \in \N$,
	we define the constant
  \begin{align}
	  \label{e:defcdk}
  c_{d,k} := \frac{\theta_d^{1-1/d} (1- 1/d)^{k-2+1/d}}{(k-1)! 2^{1-1/d}
  \theta_{d-1}},
\end{align}
	\begin{theo}
		\label{t:dhi}
		Suppose  A1 or A2 holds.
		%and $(d,k) \neq (2,1)$.
		%Set $f_0 := |A|^{-1}$.
	Let $\beta,\tau \in \R$ with $\tau >0$. Suppose $m:\N \to \N$
	with $\tau_n := m(n)/n \to \tau$ as $n \to \infty$, and for
		$n \in \N$, $t >0$ let
		\begin{align*}
			u_n := \Pr [n \theta_d f_0 R_{n,m(n),k}^d 
		- (2-2/d) \log n - (2k -4 +2/d) \log \log n) \leq \beta ];
			\\
			u'_t := \Pr [t \theta_d f_0 (R'_{t,\tau t,k})^d 
		- (2-2/d) \log t - (2k-4 + 2/d) \log \log t) \leq \beta ].
		\end{align*}
			If $(d,k) = (2,2)$ then
			as $n \to \infty$,
			\begin{align}
				u_n = \exp \Big( -
				\frac{\tau_n \pi^{1/2}\sigma_A e^{-\beta/2} \log \log n}{
					8 \log n}
					\Big)
					\exp \left( - \tau_n  \left(
					e^{-\beta} +
				\frac{\pi^{1/2} \sigma_A e^{-\beta/2}}{4}\right) \right) + 
				O \big( \frac{1}{ \log n}
			 \big),
				\label{e:lim22}
			\end{align}
			and as $t \to \infty$,
			\begin{align}
				u'_t = \exp \Big( - 
				\frac{\tau_n \pi^{1/2}\sigma_A e^{-\beta/2} \log 
				\log t}{
					8 \log t}
					\Big)
					\exp \left( - \tau  \left(
					e^{-\beta} +
				\frac{\pi^{1/2} \sigma_A e^{-\beta/2}}{4}\right) \right) + 
				O \big( \frac{1}{ \log t}
			 \big).
				\label{e:Polim22}
			\end{align}
			%If $d=2, k \geq 3$ then as $n \to \infty$,
	%\begin{align*}
		%u_n = 
		%\Big(1 - \frac{\tau_n \sigma_A e^{-\beta/2} \pi^{1/2}
		%(2k-3)^2 \log \log n}{(k-1)!2^{k+1} \log n} \Big)
		%\exp \Big( \frac{-\tau_n \sigma_A e^{-\beta/2} \pi^{1/2}}{
			%(k-1)!2^k}\Big)  + O \Big( \frac{1}{\log n}\Big)
	%\end{align*}
		%and as $t \to \infty$,
	%\begin{align*}
		%u'_t = 
		%\Big(1 - \frac{\tau \sigma_A e^{-\beta/2} \pi^{1/2}
		%(2k-3)^2 \log \log t}{(k-1)!2^{k+1} \log t} \Big)
		%\exp \Big( \frac{-\tau \sigma_A e^{-\beta/2} \pi^{1/2}}{
			%(k-1)!2^k}\Big)  + O \Big( \frac{1}{\log t}\Big)
	%\end{align*}
		If $d=2, k \geq 3$ or if $d \geq 3$ then as $n \to \infty$,
	\begin{align}
		u_n =
		\exp \Big(  - \frac{ c_{d,k} \tau_n \sigma_A e^{-\beta/2}
			(k-2+1/d)^2   
		\log \log n}{
			(1-1/d) \log n}
			\Big) \exp \left(- c_{d,k} 
			\tau_n \sigma_A e^{-\beta/2} 
		\right) \nonumber \\  + O\Big(\frac{1}{\log n}\Big)
		\label{e:limgen}
	\end{align}
	and as $t \to \infty$,
	\begin{align}
		u'_t =
		\exp \Big(  - \frac{ c_{d,k} \tau \sigma_A e^{-\beta/2}
			(k-2+1/d)^2   
		\log \log t}{
			(1-1/d) \log t}
			\Big) \exp \left(- c_{d,k} 
			\tau \sigma_A e^{-\beta/2} 
		\right)  \nonumber \\
		+ O\Big(\frac{1}{\log t}\Big).
		\label{e:Polimgen}
	\end{align}
	%\begin{align}
	%	\lim_{n \to \infty}
	%	\Pr [n \theta_d f_0 R_{n,m(n),k}^d 
	%	- (2-2/d) \log n + (4-2k-2/d) \log \log n) \leq \beta ]
	%	\nonumber
	%	\\
	%	\lim_{t \to \infty}
	%	\Pr [t \theta_d f_0 (R'_{t,\tau t,k})^d -
	%	(2 - 2/d) \log t+ (4-2k-2/d) 
	%	\log \log t) \leq \beta ]
	%	\nonumber
	%	\\
	%= \begin{cases}
	%	\exp( - \tau  
	%	( e^{-\beta} +(1/4) 
	%	\pi^{1/2} \sigma_A e^{-\beta/2}) )
	%	& ~~ {\rm if} ~ (d,k)=(2,2)
	%	\\
	%%	\exp( - \tau  (k-1)!^{-1} 2^{-k }
	%%	\pi^{1/2}   \sigma_A 
	%%	e^{-\beta/2})
	%	 \exp( - c_{d,k} \sigma_A \tau e^{-\beta/2} ).
	%	& ~~ {\rm otherwise} .
	%\end{cases}
	%	= \exp( - c_{d,k} \sigma_A \tau e^{-\beta/2} ).
		%\label{e:0831a}
	%\end{align}
	\end{theo}
	%Next we give a result for $R_{n,m,k}$ when $ k \geq 2$.
	%We consider only the case with $d=2$.
	%{\bf Maybe consider higher $d$ too? See last section.}
	%\begin{theo}[Fluctuations of $R_{n,m,k}$ when $d=2$, $k \geq 2$]
%\label{t:kbig}
%Suppose that $d = 2$  and A1 or A2 holds.
%		Set $f_0 := |A|^{-1}$.
%	Let $\beta,\tau \in \R$ with $\tau >0$. Suppose $m:\N \to \N$
%	with $m(n)/n \to \tau$ as $n \to \infty$.  Then
%\bea
%	\lim_{n \to \infty}[\Pr[n \pi f_0 (R_{n,m(n),k})^2 - \log n
%	- (2k-3) \log \log n \leq \beta] \nonumber \\
%	= \lim_{t \to \infty}[\Pr[t\pi f_0 (R'_{t,\tau t,k})^2 - \log t
%	- (2k-3) \log \log t \leq \beta] \nonumber \\
%	= \begin{cases}
%		\exp( - \tau  
%		( e^{-\beta} +(1/4) 
%		\pi^{1/2} \sigma_A e^{-\beta/2}) )
%		& ~~ {\rm if} ~ k=2
%		\\
%		\exp( - \tau  (k-1)!^{-1} 2^{-k }
%		\pi^{1/2}   \sigma_A 
%		e^{-\beta/2})
%		& ~~ {\rm if} ~ k >2.
%	\end{cases}
%	\label{eq:khi}
%	\eea
%\end{theo}

	%\begin{remark}
		%{\rm
		{\bf Remarks.}
		1. It follows from 
		\eqref{e:limgen},
		\eqref{e:Polimgen}
		that
		when $d =2, k \geq 3$ or $d \geq 3$ we have
		as $n \to \infty$ that
		$$
		n \theta_d f_0 R_{n,m(n),k}^d - (2-2/d) \log n
		- (2k-4+2/d) \log \log n \tod \Gum_{\log(c_{d,k} \tau
		\sigma_A),2},
		$$
		along with a similar result for $R_{t,\tau t,k}$. 
		On the other hand, when $d=2,k=2$ we have
		from \eqref{e:lim22} 
		%\eqref{e:Polim22}
		that
		$$
		n \pi f_0 R_{n,m(n),2}^2 - \log n  - \log \log n
		\tod \max(\Gum_{\log \tau,1},\Gum'_{\log(\tau \pi^{1/2}
		\sigma_A/4),2}),
		$$
		where $\Gum$ and $\Gum'$ denote two independent
		Gumbel variables with the parameters shown.
		The distribution of the maximum of two independent Gumbel
		variables with different scale parameters is known
		as a {\em two-component extreme value} (TCEV) distribution
		in the hydrology literature \cite{Rossi}.

		2. As in the case of Theorem \ref{Hallthm}, when $d \geq 3$
		in Theorem \ref{t:dhi}
		we could replace the $O((\log n)^{-1})$
		remainder in \eqref{e:limgen} 
		with an explicit $(\log n)^{-1}$ term
		and an $O((\frac{\log \log  n}{\log n})^2)$ remainder,
		and likewise for the $O((\log t)^{-1})$ remainder in
		\eqref{e:Polimgen};
		see \eqref{e:dkbigdet} and \eqref{e:dkPobigdet}
		in the proof of Theorem \ref{t:dhi} for details. \\
		%} \end{remark}

Comparing these results with the corresponding results
for the complete coverage threshold \cite{HallZW,Janson,Pen22}, we find that
the typical value of that threshold (raised to the power $d$
and then multiplied by $n$) is greater than the typical
value of our two-sample coverage  threshold (transformed the same way)
by a constant multiple of $\log \log n$.  For example, our
Theorem \ref{Hallthm}
has a coefficient of $k-1$ for $\log \log n$ while
 \cite[Proposition 2.4]{Pen22} has a coefficient of $k$.
 When $d=2, k=1$,
 our Theorem \ref{thsmoothgen} has a coefficient of zero for $\log \log n$
 whereas \cite[Theorem 2.2]{Pen22} has a coefficient of $1/2$.

	We shall prove our theorems using the following strategy.
	Fix $k \in \N$.
	Given $t,r >0$ define the random `vacant' set
	\begin{align}
		\label{e:Vdef}
		V_{t,r,k}:=  \{x \in A: \Po_t (B(x,r)) < k\}.
		%~~~~~ V_{t,r} := V_{t,r,1}.
	\end{align}
	%If $k=1$ write simply $V_{t,r}$ for $V_{t,r,1}$.

	Given $\gamma  \in (0,\infty)$, suppose we can find $(r_t)_{t >0}$
	such that $t \E[|V_{t,r_t,k} \cap B|]/|B|  = \gamma$.
	%By a second moment calculation (Lemma
	%\ref{l:toP}) we  an improve this convergence of expectation
	%to convergence in probability. But then we can condition
	%on the `vacant set'  $V_{t,r_t,k}$.
	If we know $t |V_{t,r_t,k} \cap B|/|B|
	\approx \gamma $, then the 
	%conditional
	distribution of
	$\cQ_{\tau t,B} (V_{t,r_t,k})$ 
	%(given $\Po_t$)
	is approximately Poisson
	with mean $\tau \gamma$,
	and we use  the Chen-Stein method
	to make this Poisson approximation quantitative, and hence
	show  that $\Pr[R'_{t,\tau t,k}
	\leq r_t]$ approximates to $e^{- \tau \gamma}$ 
	for $t $ large
	(see Lemma \ref{l:Poapprox}). By coupling binomial and Poisson
	point processes, we obtain a similar result for
	$\Pr[R_{n,m(n),k} \leq r_n]$ (see Lemma \ref{l:Rmeta}).

	Finally, we need  to
	find nice limiting expression for $r_t$ as $t \to \infty$.
	By Fubini's theorem
	$\E[|V_{t,r_t,k} \cap B|]= \int_B p_t(x)dx$,
	where we set $p_t(x)= \Pr[x \in V_{t,r_t,k}]$.
	Hence we need to take $r_t \to 0$. 
	Under A3, for $t$ large $p_t(x) $  is constant over $x \in B$ so 
	finding a limiting expression for  $r_t$ in that case is
	fairly straightforward.

	Under Assumption A1 or A2, we need to deal with boundary effects since
	$p_t(x)$ is larger for $x$ near the boundary of $A$ than in
	the interior (or `bulk') of $A$. In
	Lemma \ref{lemexp} we determine the asymptotic behaviour of
	the integral near a flat boundary; since the contribution
	of corners turns out to be negligible this enables us
	to handle the boundary contribution under A2.

	Under Assumption A1,
	we need to deal with integrals of $p_t(x)$
	over $x$ near %in a region $\partial A^{(r_t)}$
	%of $A$ within distance $r_t$ of
	the curved boundary
	of $A$.  We approximate $p_t(x)$ by a function
	depending only on $\dist(x,\partial A) := \inf_{y \in \partial A}
	\|x-y\|$, and parametrise
	$x$ by the nearest point in $\partial A$ and the distance
	from $\partial A$.
	In Proposition \ref{thm:cov} 
	we provide a useful estimate on the Jacobian arising from
	this parametrization.
	The upshot is that we can reduce the integral to a one-dimensional
	integral that can be dealt with using Lemma \ref{lemexp}.

	Alternatively it is possible to handle the curved boundary by 
	 adapting  methodology of  \cite{Pen22}, whereby one 
	 approximates $A$  by a polytope
	$A_t$ with
	spacing that tends to zero more slowly than $r_t$.
	In an earlier version of this paper (v1 on ArXiv) this alternative
	method is  carried out. However the method developed
	here,
	using Proposition \ref{thm:cov}, seems
	to provide a cleaner proof and is likely to be useful in
	other settings.

	It turns out that $d=2,k=1$ is a special case because 
	in this case only, the contribution of the bulk dominates
	the contribution of the boundary region to $\E[|V_{t,r_t,k}|]$.
	When $d=2,k=2$ both contributions are equally important,
	and in all other cases the boundary contribution dominates
	the contribution of the bulk. This is why the formula for
	the centring constant for $R'_{t,\tau t,k}$ or
	$R_{n,m,k}$ in terms of $d$ and $k$
	is different for Theorem \ref{thsmoothgen} than for Theorem
	\ref{t:dhi} (the coefficient of $\log \log n$ being 0 rather
	than 1 in Theorem \ref{thsmoothgen}),
	and why in Theorem \ref{t:dhi} 
	%the formula for
	the limiting
	distribution is TCEV 
	%involves both $|A|$ and $\sigma_A$
	for $d=k=2$ but is Gumbel
	%but involves only $\sigma_A$
	for all other cases.

 \section{Preparatory results}
 \label{s:prep}
 \allco

 We use the following notation from time to time.
 Given $r >0$, and  $A \subset \R^d$, set $\partial A^{(r)} :=
 A \cap \cup_{x \in \partial A} B_r(x)^o$. Aso
 %Given also $A \subset \R^d$,
 set $A^{(-r)}:= A \setminus \partial A^{(r)}$.
	
	%Throughout this section we assume that
	%$B \subset A$ (possibly $B=A$), and
	%A1, A2 or A3 holds. 

        Let $\pi: \R^d \to \R^{d-1} $ denote projection onto
        the first $d-1$ coordinates and let
	 $\ee_d:= (0,\ldots,0,1)$, the $d$th coordinate vector in $\R^d$. 
	 Let $x \cdot y$ denote the Euclidean inner product
	 of vectors $x,y \in \R^d$.
	%$h:\R^d \to \R$ denote
        %projection onto the last coordinate ($h$ stands for `height').
	For $a \in[0,1]$, let  
	\begin{align}
		h(a) := |B_1(o) \cap ([0,a] \times \R^{d-1})|.
		\label{e:defh}
	\end{align}
	%the volume of a slice of the unit radius ball of thickness $a$.
	%the volume of a slice of the unit radius ball of thickness $a$.
	We suppress the dependence of $h(a)$ on on the dimension $d$. 

	Throughout this section we assume that $A \subset \R^d$ is
	bounded with a $C^{1,1}$ boundary, and that $A = \overline{A^\circ}$.

\subsection{Geometrical lemmas}
\label{ss:geom}

\begin{definition}[Sphere condition]
	For $z \in \partial A$ let $\hat n_z$ be the unit normal to $\partial A$ at $z$ pointing inside $A$.

	Given $\tau \geq 0$,
	let us say $\tau $ satisfies the {\em sphere condition}
	 for $A$ if, for all $x \in \partial A$,
	we have $B(x+ \tau \hat n_x , \tau) \subset A$
	and  $B(x -\tau \hat n_x, \tau) \cap A = \{x\}$.

Let $\tau(A)$ denote the supremum of the set of all $\tau$ satisfying 
the sphere condition for $A$.
\end{definition}
\begin{lemm}[Sphere condition lemma]
	$\tau(A) >0$; that is, 
	there exists a constant $\tau >0$ such that
%	for all $x \in \partial A$, 
%	There exists a constant $\tau'= \tau'(A)$ such that
%	if $x \in \partial A$, then 
%	[in this case we say
	$\tau$ satisfies the sphere condition
	 for $A$.
\end{lemm}
\begin{proof}
	See	\cite[Lemma 7]{sphere-condition-lp}.
	%{\bf [more detailed reference?]}
\end{proof}
\begin{remark}
	{\rm 
	(i) If $0 < \tau < \tau'$ and $\tau'$ satisfies the sphere condition
 for $A$, then so does  $\tau$.

	(ii) If $x \in \R^d$ with
	$\dist(x ,\partial A) < \tau(A)$,
	then $x$ has a unique closest point in $\partial A$.
	}
\end{remark}

Given small
$r  >0  $, and
$x \in \partial A^{(r)}$, 
%be a point at distance $s \geq 0$ from $\partial A$,
 we are interested in estimating the volume of $A \cap B(x,r)$.
Using the sphere condition we can approximate this volume with
that
of a certain `sliced ball'.

%\begin{figure}
%	\centering
%	\includegraphics[width=1.0\linewidth]{ball-volume}
%	\caption{\label{fig:inner-and-outer-spheres}
%		I'll draw a proper version of this later.}
%\end{figure}

%For $a \in [0,1]$, let $h(a) := | B_1(o) \cap ([0,a]\times \R^{d-1}) |$,
%as in Section \ref{ss:keylem}
%{\bf [$h$ is something different in Section \ref{sec:change-of-variables}.]}
For $x\in A$ let $a(x): = \dist (x,\partial A)$,
the Euclidean distance from $x$ to
$\partial A$.
For $x \in \partial A^{(r)}$, we shall approximate
$| B_r(x) \cap A | $ by $ 
(\frac{1}{2}\theta_d + h(a(x)/r))r^d$,
the volume of the portion of $B_r(x)$
which lies on one side of the tangent plane to $\partial A$
at the closest point to $x$ on $\partial A$.

\begin{lemm}
	\label{l:bdyvol}
	%Let $\tau = \tau (A)$. Then
	%%Suppose $A$ satisfies the sphere condition
	%%with a sphere of radius $\tau$.
	%as $r \downarrow 0$,
	%\[
	%	\sup_{x \in \partial A^{(r)}}
	%	\left|
	%	  |B_r(x) \cap A|
	%	  - r^{d+1} ((\theta_d/2) + h(a(x)/r))
	%	\right|
	%	= O(r^{d+1}).
	%\]
	%In fact, 
	Suppose $0 < r <  \tau(A)$, 
	and $x \in \partial A^{(r)}$. Then
	\begin{equation}
		\label{eq:ball-volume-estimate}
		\left|
		|B_r(x) \cap A| - ((\theta_d/2) + h(a(x)/r))r^d
		\right|
		\leq
		\frac{2 \theta_{d-1} r^{d+1}}{\tau(A)}.
	\end{equation}
\end{lemm}

\begin{proof}
Without loss of generality the closest point on the boundary to $x$
is the origin $o$ and $x = a e_d$
for $a = a(x) \in [0,r)$.
Let $\bH := \{ y \in B_r(x) : y \cdot e_d \geq 0 \}$
the upper half-space,
	and note that $| B_r(x) \cap \bH  | = ((\theta_d/2) + h(a/r))r^d$,
the volume we are using to approximate $|B_r(x) \cap A|$.

	Let $\tau \in (r, \tau(A))$.
	Let $S:= B_\tau(\tau e_d)^o$ and $S' := B_\tau(-\tau e_d)^o$.
	Then the set $(B_r(x) \cap A) \triangle (B_r(x) \cap \bH)$
	is contained in $\R^d \setminus (S \cup S')$. Therefore
	by some spherical geometry, it is contained in
	a cylinder $C$ centred on $o$ of radius $r$ and height
	$2 s$, as illustrated in Figure \ref{f:hourglass},
	with $s $ chosen so $s \leq r$ and
	$(\tau-s)^2 + r^2= \tau^2$, so $2 \tau s= r^2 + s^2 \leq 2 r^2$,
	and hence $s \leq r^2/\tau$. Thus $|C| \leq
	2 \theta_{d-1}r^{d+1}/\tau$, and \eqref{eq:ball-volume-estimate}
		 follows by letting $\tau \uparrow \tau(A)$.  
	%{\bf [PIC?]} 
\end{proof}

	\begin{figure}[!h]

\center
\includegraphics[width=5cm]{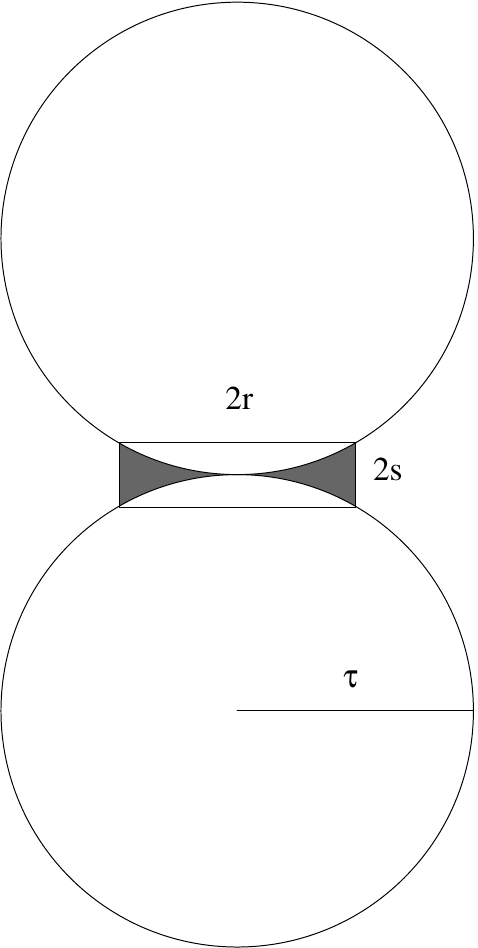}
%\includegraphics[width=1.2\textwidth]{pic1.pdf}
		%\vspace{-2cm}
		\caption{\label{f:hourglass}Illustration for proof of Lemma
		\ref{l:bdyvol}. The set $(B_r(x) \cap A)  \triangle
		(B_r(x) \cap \mathbb{H})$ is contained in the shaded region.
        }
\end{figure}

In Lemma \ref{lemgeom1a} below we give a lower bound on
the volume within $A$ of
%a ball, and of
the difference between
two balls, having their centres near the boundary of $A$.

\begin{lemm}
        \label{l:Fubex}
        %There is a constant $\kappa_d >0$ such that 
	For
        any compact convex $F \subset \R^d$ containing a Euclidean ball of
        radius $1/4$, any unit vector $\ee$ in $\R^d$, and
	any $a \in (0,2]$ we have
	$|(F + ae) \setminus F| \geq 8^{-d} \theta_{d-1} a$.
\end{lemm}
\begin{proof}
	Without loss of generality, $F \supset B(o,1/4)$ and 
	$e= e_d$.
	%$e=(0,\ldots,0,1)$.
	By Fubini's theorem,
\begin{align*}
        |(F + ae) \setminus F| \ge \int_{\pi(B(o,1/8))} \int \1\{
		(u,t) \in F, (u,t+a)\notin F\} dt du.
\end{align*}
For any fixed $u\in \pi(B(o,1/8))$, the set of $t$ such that the indicator is
	1 is an interval of length at least $\min(a,1/4)$.
	Hence, the double integral is bounded from below by $\min(a,1/4) \theta_{d-1}8^{1-d}$. The result follows.
\end{proof}

\begin{lemm}
	\label{lemgeom1a} 
	%Suppose $A \subset \R^d$ is compact with $C^2$ boundary, and
	%$A = \overline{A^o}$.
%
%	(i) Given $\eps > 0$,
%	there  exists 
%	$r_0 = r_0(d,A, \eps)>0$  such that
%	\bea
%	|  B_r(x) \cap A| \geq ((\theta_d/2)-\eps)  r^d,
%	~~~~ \forall x \in A, r \in (0,r_0).
%	\label{e:0430a}
%	\eea
%	(ii)
	%There is a constant $r_1 = r_1 (d,A)$, such that if
	If $r \in (0,\tau(A)/192)$ and $x, y \in A$
	with  $\|y-x \| \leq 3r$ and
	 $\dist(x,\partial A) \leq \dist(y,\partial A)$, then
	\bea
	|   A \cap B_r(y) \setminus B_r(x)| \geq  8^{-d}
	\theta_{d-1} r^{d-1} \|y-x\|.
	\label{e:diffball}
	\eea
	%where $\kappa_d$ is given in Lemma \ref{l:Fubex}.
\end{lemm}
\begin{proof} 
	It suffices to consider the case with $x \in \partial A^{(r)}
	\cap A$.
	%(for both (i) and (ii)).
%
	%It follows from Lemma \ref{l:bdyvol} that if $r < \tau(A)/2$,
	%we have $|B_r(x) \cap A| \geq \frac12 \theta_d r^d - 4 
	%\theta_{d-1} r^{d+1} \tau(A)^{-1} $, and hence for 
	% Part (i) we can take $r_0 = \eps \tau(A)/(4 \theta_{d-1})$. 
%
	%Choose $\delta \in (0,1/2)$
	%such that we have $|\{x \in B_1(o): h(x) > \delta\}|
	%\geq  (\theta_d/2) -\eps$.
%
%	For (ii), 
	Let $x \in \partial A^{(r)} \cap A $. 
	%
	%------------
%
	%Let $x \in (\partial A)^{(r)} \cap A$ with
	%$K_0$ to be chosen later.
	Without loss of
	generality (after a rotation and translation),
	we can assume that the closest point of $\partial A$ to $x$  
	lies at the origin, and $x = \|x\| e_{d}$.
	%(where $e_d$ is the $d$th coordinate vector),
	%and for some convex open $V \subset \R^d$ with $o \in V$
	%and some open convex neighbourhood $U$ of the
	%origin in $\R^{d-1}$, 
	%and some $C^2$ function
	%$\phi :U \to\R$ we have that $A \cap V
	%= V \cap \mathrm{epi}(\phi) $,
%where $\mathrm{epi}(\phi):= \{(u,z) \in U \times \R: z \ge \phi(u)\}$,
%the closed epigraph of $\phi$.

	Fix $\tau \in (0,\tau(A))$. 
	Since $z=o$ is the closest point in $\partial A$ to $x$,
	$\hat n_o = e_d$, so
	by the sphere condition
	$B_\tau(\tau e_d) \subset A$ and
	$B_\tau(-\tau e_d)^o \subset A^c$. For $u \in \R^{d-1} $ with
	$\|u\| < \tau$, define
	$$
	\phi(u):= \sup \{a \in [-\tau,\tau]: (u,a) \notin A\}.
	$$
	Then $\phi(u) \leq s(\|u\|)$ where for $0 \leq v < \tau$
	we define  $s(v)$ so $0 \leq s(v) \leq v$ and $(\tau - s(v))^2 + v^2
	= \tau^2$, and hence $s(v) \leq v^2/\tau$ as in the proof
	of Lemma \ref{l:bdyvol}.
	%we must have $\nabla \phi(o) =o$. By a compactness argument,
	%we can also assume
	%$\sum_{i=1}^d \sum_{j=1}^d |\partial^2_{ij} \phi| \leq K/(99d^2)$
	%on $U$ for some constant $K$ (depending on $A$).

	Now suppose $0 < r < \tau/4$.
	%and let $u \in \R^{d-1}$ with $\|u \| \leq 4r$.
	Set $K= 16/\tau$. Then
	%(assume $r$ is small enough that all such $u$ lie
	%in $U$). 
	%By the Mean Value theorem
	%$\phi(u) =  u \cdot \nabla \phi(w)$ for some $w \in [o,u]$,
	%and for $1 \leq i \leq d$, $\partial_i \phi(w) =
	%w \cdot \nabla \partial_i \phi(v)$ for some
	%$v \in [o,w]$. Hence 
	\bea
	|\phi(u)| \leq  %(K/99)
	\tau^{-1} \|u\|^2 \leq
	K r^2, ~~~ \forall u \in \R^{d-1} ~{\rm with}~ \|u\| \leq 4r.
	\label{0127b2}
	\eea

%	For $z \in H$ we have $\|\pi(z)\| \leq r$ so
%	that by (\ref{0127b2}), $\phi(\pi(z)) \leq Kr^2$.
%	On the other hand $h(z) \geq  r$, so provided
%	$r \leq \delta/K$, using (\ref{0127b2})
%	we have $\phi(\pi(z)) \leq Kr^2 \leq \delta r \leq h(z)$
%	and therefore $H \subset A$. Hence
%	$
%	|B_r(x) \cap A | \geq |H | \geq ((\theta_d/2)-\eps) r^d,
%	$
%	and thus \eqref{e:0430a} holds on taking $r_0= \delta/K$.
%	Thus we have Part  (i).
%
	%Now 
	Let $y \in B_{3r}(x) \cap A \setminus \{x\}$
	with $\dist(y,\partial A) \geq
	\dist(x,\partial A)$.
	We need to find a lower bound on $| A \cap B_r(y) \setminus B_r(x)|$.

	First suppose $y \cdot e_d \geq x \cdot e_d$.
	Let $H:= \{z \in B_r(x): (z-x)\cdot e_d \geq   r/4 \}$.
	We claim $H + (y-x) \subset A $. Indeed,
	for $z \in H + (y-x) $ we have $\| \pi(z)\| \leq 4 r$,
	and hence $\phi(\pi(z)) \leq Kr^2$ by (\ref{0127b2}). 
	Therefore, provided $r<1/(4K)$,
	we have
	$z \cdot e_d  \geq  r/4 \geq Kr^2 \geq \phi(\pi(z))$, so $z \in A$,
	justifying the claim. Using the claim, and Lemma \ref{l:Fubex},
	we obtain that
	\begin{align}
		| A \cap  B_r(y) \setminus B_r(x) | & \geq
	|(H +(y-x))  \setminus H| 
		\nonumber \\
		& \geq 8^{-d} \theta_{d-1}
	r^{d-1} \|y-x\|,
	~{\rm if}~ y \cdot e_d \geq x \cdot e_d.
	\label{0127c2}
	\end{align}

	Now suppose $y \cdot e_d < x \cdot e_d$.
	Note that
	$\pi(y) \neq \pi(x)$ since $y \neq x$ and
	$\dist(y,\partial A) \geq \dist(x,\partial A)$.
	Let $B'_r$ be the closed half-ball of radius
	$r$ centred on $x$, having the property
	that $y' := x+ (r /\|y-x\|) (y-x)$ has the
	lowest $d$-coordinate of all points in $B'_r$.
	Let $H'$ be the portion of $B_r(x)$ lying above
	the upward translate of the bounding hyperplane of $B'_r$
	by a perpendicular  distance of	$r/4$
	(see Figure \ref{fig2}).
	\begin{figure}[!h]

\center
\includegraphics[width=10cm]{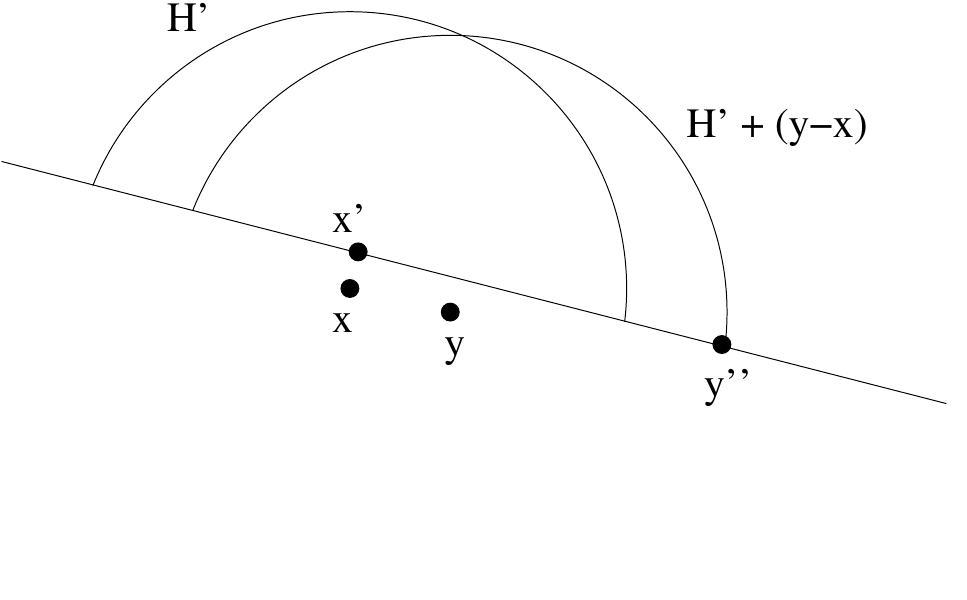}
		\vspace{-2cm}
		\caption{\label{fig2}Illustration for proof of Lemma
		\ref{lemgeom1a}. The segment $H'$ is centred on $x$.
        }
\end{figure}

	Since $\dist(y,\partial A) \geq \dist(x, \partial A) =
	x \cdot e_d$,
	using (\ref{0127b2})
	we have 
	\bea
	y \cdot e_d \geq \phi(\pi(y)) + x \cdot e_d  \geq x \cdot e_d 
	- (K/9)\|\pi(y)\|^2. 
	\label{0128a}
	\eea
	Let $x'$ be the point in the bounding hyperplane
	of $H'$ that lies closest to $x$. Then the line
	segment $[x,x']$ is almost vertical;
	the angle between this line segment and the
	vertical is the
	same as that between the line segment $[x,y]$ and the horizontal.
	Therefore
	\bea
	\frac{(x' - x) \cdot e_d}{r/4}  = \frac{\|\pi(y)\|}{\|y-x\|}.
	\label{0128b}
	\eea
	Using (\ref{0128a}), provided $r < 1/K$
	we have 
	$$
	\|y-x\| \leq \|\pi(y)\| (1 + (K/9) \|\pi(y)\|)   
	\leq (9/8) \|\pi(y)\|,
	$$
	so we obtain from (\ref{0128b}) that
	$$
	(x' - x) \cdot e_d \geq (2/9) r.
	$$
	Now letting $y''$ be the  point in $H' + (y-x)$
	with lowest $d$-coordinate, we have that 
	$y'' = x' + a (y-x)$
	%for some $a \in [0,1 + r/\|y-x\|]$.
	with $a = 1 +  (15/16)^{1/2}r/\|y-x\|$ (note  $H'$
	is not quite a half-ball).
	Hence
	$$
	(x'- y'') \cdot e_d  \leq \left( \frac{3r}{\|y-x\|} \right)
	(x-y) \cdot e_d \leq K r^2,
	$$
	where the last inequality came from (\ref{0128a}).
	Hence, provided $r < 1 /(12K) = \tau/192$,
	\bean
	%h(y'') \geq h(x') - 2 (h(x) - h(y))
	y''\cdot e_d = x'\cdot e_d -  (x' - y'') \cdot e_d
	\geq x \cdot e_d + (2/9)  r - K r^2  \geq
	 r/8.
	\eean
	On the other hand, for all $z \in H'+ (y-x)$ we
	have $\|\pi(z) \| \leq 4r$,  so that $\phi(\pi(z)) \leq
	Kr^2$ by (\ref{0127b2}). Provided $r < 1/(8K)$ 
	we therefore
	have $
	\phi(\pi(z))
	\leq r/8 \leq 
	z \cdot e_d $ and hence $z \in A$. Therefore $H' + (y-x) \subset A$.
	Also $H'$ contains a ball of radius $r/4$.
	%since $\delta  < 1/2$.
	Therefore using Lemma \ref{l:Fubex},
	we have
	$$
	|A \cap B_r(y) \setminus B_r(x) | \geq
	|(H' + (y-x) ) \setminus H'| \geq 8^{-d} \theta_{d-1}
	r^{d-1}\|y-x\|
	,~~{\rm if~} y\cdot e_d < x \cdot e_d.
	$$
	%where the constant $\eps_3 >0$ depends only on $d$.
	Combined
	with (\ref{0127c2}) this yields 
	(\ref{e:diffball}).
	%completing the proof of Part (ii).
	%Combining that with (\ref{0128c}) yields (\ref{2ptlb1a}).
\end{proof}

%\subsection{Sphere condition and reparameterisation}

%\begin{definition}
%	\label{def:median-axis}
%	The \emph{median axis} of a closed set $B \subseteq \R^d$ is the set
%	\[
%		\left\{ z \in \R^d : \#\{ b \in B : \|z-b\| = \inf_{b' \in B} \|z-b'\| \} > 1 \right\},
%	\]
%	i.e.\ the set of points in $\R^d \setminus B$
%	which do not have a unique closest point in $B$.
%\end{definition}
%
%For example,
%the median axis of the sphere $\partial B(x,r)$ is $\{x\}$,
%and the median axis of the square $\partial [0,1]^2 \subset \R^2$
%is the union of the two diagonals.
%
%\begin{definition}
%	\label{def:reach}
%	The \emph{reach} of a closed set $B \subseteq \R^d$
%	is the distance between $B$ and its median axis,
%	denoted $\tau(B)$.
%\end{definition}

	\subsection{Integral asymptotics}
	\label{ss:keylem}
	For $a \in[0,1]$, let  $h(a) := |B_1(o) \cap ([0,a] \times \R^{d-1})|$
	as at \eqref{e:defh}.
	%(we suppress the dependence of $h(a)$ on on the dimension $d$ 
	%from the notation). 
	%and $g(x) = h(a)/\theta_d$. 
The following lemma is very useful for estimating the integral
of $\PP[x \in V_{t,r,k}] $
over a region near the boundary of $A$,
where $V_{t,r,k}$ was defined at \eqref{e:Vdef})
\begin{lemm}
	\label{lemexp}
Let $\ell, j \in \Z_+:= \N \cup \{0\}$ and let $\alpha_0 >0$, $\eps \in (0,1)$.
	Then as $s \to\infty$,
\begin{align}
\theta_{d-1}
\int_0^1 e^{-sh(a)} (\alpha_0 + h(a))^\ell da =
\alpha_0^\ell s^{-1}
+ \ell \alpha_0^{\ell-1}  s^{-2} + O(s^{\eps -3}).
	\label{e:intest1}
	\end{align}
	Also
	\begin{align}
		\theta_{d-1}
\int_0^1 e^{-s h(a)} \left( \alpha_0 + h(a) \right)^{j}
\left( 1 + \frac{j}{s\left(\alpha_0 + h(a) \right) }
\right) da 
		= \alpha_0^j s^{-1} +  2 j\alpha_0^{j-1} s^{-2} 
	+ O(s^{\eps -3}).
		\label{e:intest2}
	\end{align}
\end{lemm}
\begin{proof}
	Note first, for $0 < x < 1$, that
	\begin{align*}
		h(x) & = \theta_{d-1} \int_0^x (1-y^2)^{(d-1)/2}dy
		= \theta_{d-1} \int_0^x
		%\left( 1 - ((d-1)/2) y^2 + O(y^4) \right) dy
		\left( 1 + O (y^2) \right) dy
		\nonumber \\ &
		 = \theta_{d-1}x  + O(x^3)
		%\nonumber \\
		%& = \theta_{d-1}(x- ((d-1)/6)
		%%\frac{d-1}{6}
		%x^3 + O(x^5))
		~~~{\rm as}~x \downarrow 0.
		%\label{e:expandh}
	\end{align*}
Thus,	setting $w=\theta_{d-1} sa$, we have
$h(a) = w/s + O((w/s)^3)$, and
	 $e^{-sh(a)} = e^{-w}(1+ O(w^3/s^2))$. 
	Given $i \in \Z_+ $,
	let $\delta = \eps/(4+i) $.
	Then
	\begin{align*}
		\theta_{d-1} \int_0^{s^{\delta-1}}  e^{-sh(a)} h(a)^i da & = 
		%\theta_{d-1}^{-1}	
		\int_0^{\theta_{d-1}s^\delta} e^{-w} \Big(1 + 
		O\big(\frac{w^3}{s^2}\big)\Big)
		\Big(\frac{w^i}{s^{i+1}}\Big) 
		\Big(1 + O\big(\frac{w^2}{s^2}\big)\Big)^i 
		%(\theta_{d-1} s)^{-1}
		dw
		\\
		& = s^{-i-1} \int_0^{\theta_{d-1}s^\delta} w^i
		e^{-w} dw + O\Big( s^{-3-i} \int_0^{\theta_{d-1}s^{\delta}}
		w^{3+i} dw \Big)
		\\
		& =  s^{-i-1} \Big(i! 
		- \int_{\theta_{d-1}s^\delta}^\infty w^i e^{-w} dw\Big)
		+ O(s^{-i-3} s^{\delta(4+i)}) 
		\\
		& = 
		%\theta_{d-1}^{-1}
		i!s^{-i-1} + O(s^{\eps-i-3}).
		\end{align*}
	Also $ \int_{s^{\delta-1}}^1  e^{-sh(a)} h(a)^i da   $ 
 is $O(e^{-(\theta_{d-1}/2)s^\delta})$ since
$ (\theta_{d-1}/2)
s^{\delta -1} \leq h(a) \leq \theta_d/2$ for $a$ in this range.
Therefore by binomial expansion, for $\ell \in \Z_+$
we have \eqref{e:intest1}.
Applying \eqref{e:intest1}
with $\ell=j$ and (if $j >0$) also
with $\ell = j-1$ gives us \eqref{e:intest2}.
\end{proof}

%\subsection{Change of variables formula}
%\label{sec:change-of-variables}
For integrating functions near the boundary of a smoothly-bounded set $A$,
	we have a useful change of variables which allows us to turn an integral
	over a region near the boundary into a double integral
	with one variable a boundary point
	and the other variable the distance to the boundary.
	
	%Inspired by something in the middle of a proof in \cite{KTV}.

\begin{prop}[Reparameterization]
	\label{thm:cov}
	There are positive finite  constants $c = c(A), r_0 = r_0(A)$, such that
%For $r \in (0,\frac{1}{2}\tau(\partial A))$,
for all $r \in (0,r_0)$, and all bounded measurable
	 %$h : \partial A^{(r)} \to [0,\infty)$,
	 $\psi: A \to [0,\infty)$,
	 %be bounded and measurable.
%
	%setting $h_s(z) := h(z + s \hat n_z)$.
%for each $s \in [0,r)$ and $z \in \partial M$, 
%we have
%
%	for all $r \in (0, \frac12 \tau(\partial A))$,
	%as $r \downarrow 0$
\begin{equation}
	\label{eq:change-of-variables}
	\Big |\int_{\partial A^{(r)}} \psi(y) \, \d y
%	= (1 + O(r))
	- \int_{0}^{r} \int_{\partial A} \psi(z + s \hat n_z) \, \d z \, 
	\d s \Big|
	\leq c r
	 \int_{0}^{r} \int_{\partial A} \psi(z + s \hat n_z) \, \d z \, \d s,
\end{equation}
where the
	%constant in the $O$ term depends only on $A$.  The 
	inner integral is a surface integral.
%
%\begin{coro}
	If $\psi(y)$ depends only on $\dist(y, \partial A)$,
	i.e.\  there exists $\Psi :[0,r_0) \to \R$ such that
	$\psi(z + s \hat n_z) = \Psi (s)$ for all $(z,s) \in \partial A \times 
	(0,r_0]$,
then
	%\eqref{eq:change-of-variables} becomes
\begin{equation}
	\Big| \int_{\partial A^{(r)}} \psi(y) \, \d y 
	- |\partial A| \int_0^{r} \Psi (s) \,\d s \Big|
	\leq c  r |\partial A| %|\partial M|
	\int_0^{r} \Psi (s) \,\d s.
	\label{e:CoVcoro}
\end{equation}
\end{prop}
%\end{coro}
\begin{proof}
	By the assumptions on $A$, for each $x \in \partial A$ 
	there is a constant $\delta(x) \in (0,\tau(A)/3)$, 
	such that 
	after a rotation $\cR$ about $x$,
	within the ball $B(x,3 \delta(x))^o$,
	the set $A$
         coincides with the closed epigraph of a
	$C^{1,1}$ function $\phi:U\to \R$
	with $U$ an open ball of radius $3 \delta(x)$
         in $\R^{d-1}$
	centred on $\pi(x)$ (recalling that
	$\pi: \R^d \to \R^{d-1} $ denotes projection onto
	the first $d-1$ coordinates); that
	is, 
	$$
	\cR(A) \cap B(x,3 \delta(x)) = \{(u,s): u \in U, s \in
	[\phi(u),\infty) \} \cap B(x,3 \delta(x)).
	$$

	 By a compactness argument we can cover $\partial A$ with a
	finite  collection of balls $B(x_i, \delta(x_i))$, $1 \leq i \leq I$
	 with $x_1, \ldots, x_I \in \partial A $.
	 For $r < \min_i \delta(x_i) $  we have $\partial A^{(r)} \subset
	 \cup_{i=1}^I B(x_i,2 \delta(x_i))$.
	 Since we can consider separately the integral of $\psi$ over 
	 $\partial A^{(r)} \cap B(x_1,2 \delta(x_1)$, over 
	 $\partial A^{(r)} \cap B(x_2,2 \delta(x_2)) \setminus
	 B(x_1,2 \delta(x_1))$,
	 over
	 $\partial A^{(r)} \cap B(x_3,2 \delta(x_3)) \setminus 
	 [B(x_1,2 \delta(x_1)) \cup  B(x_2,2 \delta(x_2))]$, and so on, 
	 it suffices to prove the result for the case where
	 $\psi$ is supported by a single ball $B(x, 2 \delta(x))$ for some
	 fixed
	 $x \in \partial A$.

	 %Then by taking an appropriate rotation we can assume
	 Without loss of generality we assume the rotation $\cR$ is
	 the identity map, so 
	within the ball $B(x,3 \delta(x))^o$,
        $A$ coincides with the closed epigraph of a
	$C^{1,1}$ function $\phi:U\to \R$
	with $U$ a $(d-1)$-dimensional open ball of radius $3 \delta(x)$
	centred on $\pi(x)$.

	\begin{figure}[!h]
\center
\begin{tikzpicture}
	\node at (0,0) {\includegraphics[width=10cm]{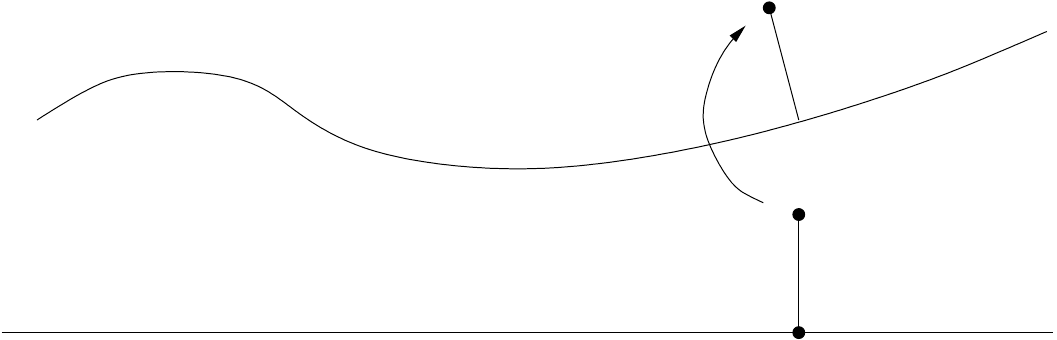}};
	\node at (-0.2,0.3) {\footnotesize $\partial A$};
	\node at (3,1.6) {\footnotesize $g(u,s)$};
	\node at (3.1,-0.4) {\footnotesize $(u,s)$};
	\node at (2.7,-1.8) {\footnotesize $(u,0)$};
\end{tikzpicture}
\caption{\label{fig3}Illustration of the mapping $g$ in the proof of Proposition
	\ref{thm:cov}.
}
\end{figure}

	For $(u,s) \in U \times (0,\delta(x))$ let $g(u,s) := (u, \phi(u) ) +
	s \hat n_{(u,\phi(u))}$, as shown in Figure \ref{fig3}, and observe that 
	$
	\hat n_{(u,\phi(u))} = (1+ | \nabla \phi|^2)^{-1/2} (-\nabla \phi, 1).$
	Since $\delta(x) < \tau(A)$, it follows
	from the sphere condition that $g: U \times (0,\delta(x)) \to A$ is
	injective. Since $\phi$ is $C^{1,1}$, $\nabla \phi(u)$ is Lipschitz
	continuous on $u \in U$, and therefore by Rademacher's theorem
	(see e.g. \cite{Fed}), there exists a set $U' \subset
	U$ of full $(d-1)$-dimensional
	Lebesgue
	measure such that $\hat n_{u,\phi(u)}$ is differentiable 
	for all $u \in U'$. Moreover by the Lipschitz continuity, 
	and the definition of partial derivatives,
	all partial derivatives of $\hat n_{u,\phi(u)}$ are uniformly bounded on
	$U'$.
	Then for $0< r < \delta(x) $,
	by \cite[Theorems 3.2.5 and 2.10.43]{Fed}
	(or if $\phi $ is $C^2$,  \cite[Theorem 17.2]{Billingsley})
	we have
	\begin{align}
		\int_{\partial A^{(r)}} \psi(y)dy = \int_{U \times (0,r)}
	\psi(g(u,s)) 
	\Big| \det \Big(\frac{\partial(g(u,s))}{\partial (u,s)}  \Big) \Big| 
	d(u,s),
		\label{e:forCoV}
	\end{align}
	where $J:= \frac{\partial g(u,s)}{\partial(u,s)}$ is the $d \times d$
	Jacobian matrix of the mapping $g$, which is defined
	for all $(u,s) \in U' \times (0,r)$.
	Given $i,j \in \{1,\ldots, d-1\}$
	the $(i,j)$th entry $J_{ij}$ of $J $ is given
	by $\frac{\partial g_i}{\partial u_j} = \delta_{ij} + O(r)$,
	where the constant in the $O$ term is independent of
	$u \in U'$ and $s \in (0,r)$. Also  $J_{dj} = \frac{
		\partial \phi}{\partial u_j} + O(r)$, while
		$J_{id} = \frac{ \partial g_i(u,s)}{\partial s}$
		so the last column of $J$ is given by the vector
		$\hat n_{(u,\phi(u))}$. Therefore
		\begin{align*}
	\Big| \det \Big(\frac{\partial(g(u,s))}{\partial (u,s)}  \Big) \Big| 
			& = (1+ O(r))(1+ | \nabla \phi|^2)^{-1/2} \Big( 1 + 
			\big(
	\frac{\partial \phi}{\partial u_1}\big)^2 +
			\cdots + \big(\frac{\partial \phi}{
		\partial u_{d-1}}\big)^2  \Big) \\
	&	= 
	 (1+ O(r) )(1+ | \nabla \phi|^2)^{1/2},
		\end{align*}
		where the $O$ term is independent of $(u,s) \in U' \times
		(0,r)$.
		Therefore by \eqref{e:forCoV},
		\begin{align*}
		\int_{A^{(r)}}
			\psi(y) dy 
			&	= (1+O(r)) \int_0^r \int_U 
		\psi(g(u,s))
			\sqrt{
				1  + |\nabla \phi(u)|^2} du ds
			\\
			&	= (1+O(r)) \int_0^r \int_U 
			\psi((u,\phi(u))+ s\hat n_{(u,\phi(u))})
			\sqrt{1  + |\nabla \phi(u)|^2} du ds
			\\
			&	= (1+O(r)) \int_0^r \int_{\partial A} 
			\psi (z + s \hat n_z) dz ds,
		\end{align*}
		which gives us 
	\eqref{eq:change-of-variables}. It is clear that \eqref{e:CoVcoro} follows from
	\eqref{eq:change-of-variables}. 
\end{proof}

%\section{Further geometrical preliminaries}
%\label{s:moreprelims}
\author{}
% \author{Frankie Higgs\thanks{\href{mailto:fh350@bath.ac.uk}{fh350@bath.ac.uk}, \href{https://people.bath.ac.uk/fh350/}{https://people.bath.ac.uk/fh350/}, ORCiD \orcidlink{0000-0002-7300-8412} 0000-0002-7300-8412} }

\section{Probability approximations}
\allco

In this section we 
assume $k \in \N$ is fixed and
$(r_t)_{t >0}$ is given and satisfies $tr_t^d = \Theta(\log t)$ as
$t \to \infty$. 
With $V_{t,r,k}$ defined at \eqref{e:Vdef},
		for $x, y \in A$ we define
		\begin{align}
		p_t(x):= \Pr[x \in V_{t,r_t,k}]; ~~~~~ \pi_t(x,y):= \Pr[\{x,y\} \subset V_{t,r_t,k}].
			\label{e:defptpit}
		\end{align}
	Since $k$ is fixed we are suppressing
	the dependence on $k$ in this notation.
For Borel $B \subset A$ with $|B|>0$, we define
 \begin{align}
	 \gamma_t(B) := (t/|B|) \E[|V_{t,r_t,k} \cap B|]
	 = (t/|B|) \int_B p_t(x) dx, 
	 \label{e:defgamma}
 \end{align}
 where the second identity in \eqref{e:defgamma} comes from Fubini's theorem.

 In Lemma \ref{l:Poapprox} below we approximate
 $\Pr[R'_{t,\tau t,k}(B) \leq r_t] $ using Poisson approximation
 (by the Chen-Stein method)
 for the number of $Y$-points lying in the region $V_{t,r_t,k}$.
 Then in Lemma \ref{l:Rmeta} we approximate 
 $\Pr[R_{n,m,k}(B) \leq r_n] $ by a suitable
 coupling of Poisson and binomial point processes.
\begin{lemm}[Poisson approximation]
	\label{l:Poapprox}
	Suppose A1, A2 or A3 holds.
	Assume that $\gamma_t(B) = O(1)$
	%\to \gamma \in (0,\infty)$ 
	as $t \to \infty$. Let $\tau \in (0,\infty)$. 
	Let $\eps >0$. Then
$$
	\sup_{\tau \in (\eps,1/\eps)}
	|\Pr[R'_{t,\tau t,k}(B) \leq r_t] - e^{-\tau \gamma_t(B)} | = O((\log t)^{1-d}).
$$
\end{lemm}
\begin{proof}
	Let $W_t := \sum_{y \in \cQ_{\tau t}} \1 \{\Po_t(B_{r_t}(y)) < k\}.$
	Then $\Pr[R'_{t,\tau t, k}(B) \leq r_t] = \Pr[W_t =0]$. 

	Let $d_{{\rm TV}}$ denote total variation distance
	(see e.g. \cite{Pen03}).
	Then
		$|\Pr[ W_t  =0 ]- e^{\tau \gamma_t(B)}|
		\leq d_{{\rm TV}}(W_t , Z_{\tau \gamma_t(B)})$.
	Hence, by a similar argument to \cite[Theorem 6.7]{Pen03},
	\begin{align*}
		|\Pr[ W_t  =0 ]- e^{\tau \gamma_t(B)}|
		\leq 3 (I_1(t) + I_2(t)),
	\end{align*}
	where, with $p_t(x)$ and $\pi_t(x,y)$ defined at
	\eqref{e:defptpit}, we set
	\begin{align}
		I_1(t) & := \tau^2 (t/|B|)^{2}  \int_B \int_{B(x,3r_t) \cap B} 
		%\Pr[\Po_t(B_r(x)) < k] \Pr[\Po_t(B_r(y)) < k]
		p_t(x) p_t(y)
		 dy dx;
		\label{e:defI1}
		\\
	%\end{align}
	%\begin{align}
		I_2(t) & :=  \tau^2 (t/|B|)^{2} \int_B \int_{B(x,3r_t) \cap B} 
		\pi_t(x,y)
		%\Pr[\{\Po_t(B_r(x)) < k \} \cap \{ \Po_t(B_r(y)) < k \}] 
		 dy dx.
		\label{e:defI2}
	\end{align}
	Define the Borel  measure $\nu$ on $\R^d$ by
	\begin{align}
	\nu(\cdot):= \lambda_d(\cdot \cap A)/|A|,
		\label{e:defnu}
	\end{align}
	where $\lambda_d$ denotes $d$-dimensional Lebesgue measure.
	Under any of A1, A2 or A3 (using Lemma 
	%\ref{lemgeom1a} 
	\ref{l:bdyvol}
	in the case of A1),
	we can and do choose $\delta >0$ such
	that for all $y \in B$ and all $r \in (0,1]$
	we have $\nu(B_r(y) ) \geq 2 \delta   r^d$. 
	Hence, for all large enough $t$
	and  all $y \in B$  we have
	$$
	p_t(y) = \sum_{j=0}^{k-1} ((t \nu(B_{r_t}(x)))^j/j!)
	e^{-t \nu(B_r(x))} \leq \exp(-\delta tr_t^d).
	$$
	Since $ (t/|B|)\int_B p_t(x) dx =
	\gamma_t(B) $ which we assume is bounded,
	we have
	\begin{align}
		I_1(t) \leq \tau^2 |B|^{-1} (t \theta_d (3r_t)^d)
		e^{-\delta tr_t^d}
		(t/|B|)
		\int_B p_t(x)   dx 
		= O(e^{-(\delta/2) t r_t^d}). 
		\label{e:I1ub}
	\end{align}
	
	Now consider $I_2(t)$.
	For $x,y \in A$ let us write $x \prec y$ if $x$ is closer than
	$y$ to $\partial A$ (in the Euclidean norm), or if $x$ and
	$y$ are the same distance
	from $\partial A$ but $x$ precedes $y$ lexicographically.
	Since $\pi_t(x,y) {\bf 1}\{\|y-x\| \leq 3r_t\}$
	is symmetric in $x$ and $y$, we have
	%(taking $B=A$ for now)
		\begin{align}
			\int_B \int_{B(x,3r_t)  \cap B } \pi_t(x,y) dy dx
			= 2 \int_B \int_{\{y \in B \cap B(x,3r_t):
			x \prec y\}}
	\pi_t(x,y) dy dx.
			\label{e:V21a}
		\end{align}
		%Set $j := k-1$.
	By the independence properties of the Poisson process
		we have
		$$
		\pi_t(x,y) \leq p_t(x) \sum_{m=0}^{k-1} q_{t,m}(y,x),
		$$
		where we set $q_{t,m}(y,x) := \Pr[\Po_t(B_{r_t}(y) \setminus
		B_{r_t}(x) ) = m].$

		Suppose Assumption A1 or Assumption A3 applies.
		Set $\kappa_d:=8^{-d}\theta_{d-1}$.
	By %\cite[Lemma 2.4]{PY23}
		Lemma \ref{lemgeom1a},
		%(ii),
		%under A1, or by the proof
		%of \cite[Proposition 3.1, Case 2]{PY23}
		%under A2,
	%	there is a constant $\kappa_d >0$ such that 
		for all large enough $t$ and all $x,y \in
	B$ with $x \prec y$ and $\|y-x\| \leq 3r_t$,
	we have $\nu(B_{r_t}(y) \setminus B_{r_t}(x))
	\geq \kappa_d f_0 r_t^{d-1} \|y-x\|$. Moreover
	by Fubini's theorem
	 $\nu(B_{r_t}(y) \setminus B_{r_t}(x))
	 \leq \theta_{d-1}f_0 r_t^{d-1}\|y-x\|$.
	Hence for all $m \leq k-1$,
$$
		q_{t,m}(y,x) \leq (t f_0 \theta_{d-1}  r_t^{d-1}\|y-x\|)^m
		\exp (- \kappa_d f_0 t r_t^{d-1} \|y-x\|).
	$$
		Hence, setting $B_x := \{y \in B: x \prec y\}$ we have that
		\begin{align*}
			t 
			\int_{B_x \cap B(x,3r_t)} q_{t,m}(y,x)
			& \leq
			f_0^m \theta_{d-1}^m 
			t
			\int_{B(o,3r_t)}
			(t r_t^{d-1} \|y\|)^m
			 \exp(-\kappa_d f_0 tr_t^{d-1} \|y\| )
			dy \nonumber \\
			& =
			f_0^m \theta_{d-1}^m 
			t (tr_t^{d-1})^{-d} \int_{B(o,3tr_t^d)} 
			\|z\|^m \exp(- \kappa_d f_0 \|z\| ) 
			dz
			\\
			& \leq c(tr_t^d)^{1-d},
		\end{align*}
		for some constant $c$ depending only on
		$d, f_0 $ and $k$.
		%of $x $ or $m$.
	Therefore
		\begin{align*}
			t^2 \int_B \int_{B(x,3r_t) \cap B} \pi_t(x,y) dy dx
			\leq  2 \left( t \int_B p_t(x) dx \right)
			ck  (tr_t^d)^{1-d}.
		\end{align*}
		Since the expression in brackets on the 
		right is $O(1)$
		%tends to a finite constant $\gamma$
		by assumption, we thus have
		$I_2(t) =O((tr_t^d)^{1-d})
		= O((\log t)^{1-d})$.

		Now suppose we assume instead that Assumption A2 applies.
		  First we examine the situation where $x$ is not too close to the corners of $A$.
%       We make more explicit the meaning of proximity to the corners as follows.
        Suppose that $\dist(x,\Phi_0(A))>Kr_t $, where 
	$\Phi_0(A)$ denotes the set of corners of $A$ and the constant
	$K$ will be made explicit later.
        We can assume that the corner of $A$ closest to $x$
        is formed by edges $e,e'$ meeting at the origin with angle
        $\alpha\in (0,2\pi) \setminus \{\pi\}$.
        We claim that, provided $K>4+8/|\sin \alpha|$,
        the disk $B(x,4r_t )$ intersects at most one of the two edges.
        Indeed, if it intersects both edges,
        then taking $w\in B(x,4r_t) \cap e, w'\in B(x,4r_t) \cap e'$
        we have $\|w-w'\|\le 8r_t $; hence $\dist(w,e')\le 8r_t $.
%       Suppose the angle of the corner is $\alpha\in (0,\pi)$ (for otherwise the claim is trivial).
        Then, $\|w\| = \dist(w,e')/|\sin \alpha|\le 8r_t /|\sin \alpha|$.
        However, $\|w\|\ge (K-4)r_t $ by the triangle inequality,
        so we arrive at a contradiction. Also, for $t$ sufficiently large,
        non-overlapping edges of $A$ are distant more than $8r_t$ from each
	other. We have
        thus shown that
        if we take $K = 5 +(8/\min_i |\sin \alpha_i|)$, where $\{\alpha_i\}$ are
        the angles of  the  corners of $A$, then for  large $t$,
	no ball of radius $4r_t$ distant at least $Kr_t$ from the corners of
	$A$ can intersect two or more edges of $A$ at the same time.

        We have $B(x,r_t)\cup B(y,r_t)\subset B(x,4r_t)$.
	Hence, the argument leading to Lemma \ref{lemgeom1a},
	%-(ii)
	shows that
	%{\bf [notation $\nu(\cdot)$?}
       % there exists $\delta_1>0$ such that
	$\nu(B(y,r_t) \setminus B(x,r_t))
	\geq \theta_{d-1} 8^{-d}  f_0  \|x-y\| r_t$.
  Using this,
        we can estimate the contribution to the double integral on the right
	side of \eqref{e:V21a} in the same way as we did under assumption A1. 

        Suppose instead that $x$ is close to a corner of $A$ and
        $\|x-y\|\le 3r_t$. The contribution to 
	the double integral on the right
	side of \eqref{e:V21a} 
	from such pairs $(x,y)$ is at most
	$c'' t^2 r_t^4  \exp(-\delta_1 tr_t^2)$
	where $c''$ depends only on $K$ and $\delta_1>0$ depends only on $A$.
	Therefore this contribution tends to zero, and
        the proof is now complete.
\end{proof}

\begin{lemm}[De-Poissonization]
	\label{l:Rmeta}
	Suppose A1, A2 or A3 holds.
	Let $m(n)$ be such that $m(n) = \Theta(n)$ as $n \to \infty$.
	Assume $\gamma_n(B) = O(1)$ 
	%\to \gamma \in (0,\infty)$
	as $n \to \infty$. 
	%Let $\tau \in (0,\infty)$. 
	Then
$$
	|\Pr[R_{n,m(n),k}(B) \leq r_n] - e^{-(m(n)/n) \gamma_n(B)} | =
	O((\log n)^{1-d}).
$$
\end{lemm}
\begin{proof}
	Write $m$ for $m(n)$. Set $n^+ := n+ n^{3/4}$, $n^- := n- n^{3/4}$
	and $m^+:= m+ m^{3/4}, m^- := m-m^{3/4}$.
	Set $$
	W := \sum_{y \in \Y_m} \1 \{\X_n(B_{r_n}(y)) < k\}; ~~~
	 W' := \sum_{y \in \cQ_{m^-,B}} \1 \{\Po_{n^+}(B_{r_n}(y)) <k\}.
	 $$
	 Set $\gamma'_n: =
		 (n^-/|B|) \E[|V_{{n^+},r_n,k} \cap B|]$, where
		 $V_{n,r,k}$ was defined at \eqref{e:Vdef}.
	 Then, with the measure $\nu$ defined at \eqref{e:defnu},
	 %{\bf [what is $\nu(\cdot) $ below?]}
	 \begin{align}
		 \gamma'_n & =
		 \frac{n^-}{|B|} 
		 \int_B \left( e^{-n \nu(B(x,r_n)) - n^{3/4} \nu(B(x,r_n))}
		 \sum_{j=0}^{k-1} (n (1+n^{-1/4}))^j\nu(B(x,r_n))^j/j!
		 \right) dx
		 \nonumber
		 \\
		 & = \gamma_n(B)(1+ O((\log n)^{1/d}n^{-1/4})).
		 \label{e:gam'}
	 \end{align}
	 By Lemma \ref{l:Poapprox}, 
		 $|\Pr[R'_{n^+,m^-,k} (B) \leq r_n]-e^{-(m^-/n^+)\gamma'_n}|
	 = O((\log n)^{1-d})$, and hence by \eqref{e:gam'},
	  $|\Pr[
		  R'_{n^+,m^-,k}(B) \leq r_n]
	  -e^{-(m^-/n^+)\gamma_n(B)}|
	 = O((\log n)^{1-d})$. 
	 Note also that $|(m/n)-(m^-/n^+)|= O(n^{-1/4})$
	 so that $|e^{-(m/n)\gamma_n(B)}-e^{-(m^-/n^+)\gamma_n(B)}|
	 = O(n^{-1/4})$.
	 Also
	 \begin{align*}
		 |	 \Pr[R_{n,m(n),k}(B) \leq r_n] - \Pr[R'_{n^+,m^-,k}(B)
		 \leq r_n] |
		 \leq \Pr[W \neq W'].
	%	 \\
	%	 \leq \Pr[E] + \E[|W -W|\1_E]
	 \end{align*}
	 We have the event inclusion
		 $\{W \neq W'\} \subset E_1 \cup E_2 \cup E_3$, where,
		 recalling the definition of $(Z_t)_{t \geq 0}$ in
		 Section \ref{SecIntro}, we define the events
	 \begin{align*}
		 E_1 & := \{Z_{m^-} \leq m \leq Z_{m^+}\}^c \cup\{Z'_{n^-} \leq n \leq
	 Z'_{n^+}\}^c;
		 \\
		 %\{W \neq W'\} \setminus E \subset 
		 E_2 & := \{ \exists y \in \cQ_{m^-}: 
		 \Po_{n^-}(B(y,r_n)) < k,
		 (\Po_{n^+} \setminus \Po_{n^-}) (B(y,r_n)) \neq
		 0\};
		 \\
		 E_3 & :=
		 \{ \exists y \in \cQ_{m^+} \setminus \cQ_{m^-}:
		 \Po_{n^-}(B(y,r_n)) < k \}.
	 \end{align*}
	 By Chebyshev's inequality $\Pr[E_1] = O(n^{-1/2})$.
	 Also by a similar calculation to \eqref{e:gam'},
	 $(n/|B|)\E[|V_{n^-,r_n,k}\cap B|] =
	 \gamma_n(B)(1+ O((\log n)n^{-1/4}))$, and
	 \begin{align*}
		 \Pr[E_2] \leq \E[|V_{n^-,r_n,k} \cap B|/|B|] m^-  (2 n^{3/4})
		   f_0 \theta_d  r_n^d = O(n^{-1/4} \log n).
	 \end{align*}
	Similarly $\Pr[E_3] = 2m^{3/4} \E[|V_{n^-,r_n,k} \cap B|/|B|] =
	O(n^{-1/4})$.
	Combining these estimates gives the result.
\end{proof}

\section{Proof of theorems}
\label{s:proofs}
%\ref{Hallthm}, \ref{th:torus} and \ref{thsmoothgen}}
\allco

\subsection{Proof of Theorem \ref{Hallthm}}
%Recall the definition of $V_{t,r,k}$ at \eqref{e:Vdef}.
Recall the definition of $\gamma_{t}(B)$ at \eqref{e:defgamma}.
For each theorem,
we need to find $(r_t)_{t \geq 0}$
such that
$\gamma_t(B)$ 
%$t \E[|V_{t,r_t,k} \cap B|]$ 
converges as $t \to \infty$; 
we can then apply
Lemmas \ref{l:Poapprox} and \ref{l:Rmeta}. 
We are ready to do this for Theorem \ref{Hallthm} without further ado. 
Recall that $f_0 :=1/|A|$. 

\begin{proof}[Proof of Theorem \ref{Hallthm}]
Fix $k \in \N$, $\beta \in \R$. For all $t>0$ define
	$r_t \in [0,\infty)$ by 
	%nonnegative numbers satisfying
\bean
	t f_0 \theta_d r_t^d  = \max(\log t + (k-1) \log \log t + \beta,0).
%~~~~~~ {\rm for~all} ~~~~ t \geq t_0,
%\label{e:rHall}
\eean
	Set $j=k-1$. Since we assume here that $\overline{B} \subset A^o$,
	for $t$ large we have for all $x \in B$ that
	$B(x,r_t) \subset A$, and hence
	 by \cite[Theorem 1.3]{LP}, $\Po_t(B(x,r_t))$
	is Poisson with parameter $t |B(x,r_t)|/|A|
	= t f_0 \theta_d r_t^d$. 
	By \eqref{e:Vdef}, $\Pr[x \in V_{t,r_t,k}] = \Pr[\Po_t(B(x,r_t)) < k]$,
	so
	as $t \to \infty$ we have
	uniformly over $x \in B$ that
	%as $t \to \infty$ we have 
	\begin{align}
	\Pr[x \in V_{t,r_t,k}]
		& = e^{-tf_0\theta_dr_t^d}( (t f_0 \theta_dr_t^d)^j/j!)  
	(1+ j (f_0 \theta_d r_t^d)^{-1} + O((\log t)^{-2}))
		\label{e:fromPo}
		\\
		& = (1/j!) t^{-1} (\log t)^{-j} e^{-\beta}
		(\log t + j \log \log t + \beta)^j
		\nonumber
		\\
		& ~~~~~~~~~~~~~~\times 
		(1+ j (\log t + j \log \log t + \beta)^{-1} +O((\log t)^{-2}))
		\nonumber
		\\
		& = \frac{e^{-\beta}}{j!t}
		\Big( 1+ \frac{j \log \log t + \beta}{\log t} \Big)^j
		\Big(1 + j (\log t)^{-1} \Big( 1+ \frac{j \log \log t + \beta}{
			\log t} \Big)^{-1}
			\nonumber \\
		&~~~~~~~~~~~~~~~~~~~~~~~~~~~~~~~
		~~~~~~~~~~~~~~~~~~~~~~~~~~~~~~
		+ O \big((\log t)^{-2} \big) \Big),
		\nonumber
	\end{align}
	where the $O(\cdot)$ term is zero for $k=1$ or $k=2$.
	Using \eqref{e:defgamma},
	we obtain
	by standard power series expansion that
	%Also
	%$e^{-tf_0r_t^d}=t^{-1}(\log t)^{1-k}e^{-\beta}$, so
as $t \to \infty$
we have
	\begin{align}
		%t \E[  |V_{t,k,r_t} \cap B|] 
		\gamma_t(B) & 
		= \frac{e^{-\beta}}{j!} 
		\Big(  1 + \frac{j^2 \log \log t}{\log t}
		+ \frac{j(1+\beta)}{\log t} + O\big( \big( 
		\frac{\log \log t}{\log t}
		\big)^2 \big)
		\Big). \label{e:power}
		%\\
		%&	
		%\sim t |B| ((t f_0 \theta_d r_t^d)^{k-1}/(k-1)!) e^{-t f_0 \theta_d r_t^d}  
	%\nonumber	\\
	%	& \to |B|
	%e^{-\beta}/(k-1)!
	\end{align}
	Hence by Lemma \ref{l:Poapprox},
	given $\tau \in (0,\infty)$ as $t \to \infty$
	 we have
	\begin{align}
		\Pr[R'_{t,\tau t,k}(B) \leq r_t ]   =
		\exp \Big(-\frac{\tau e^{-\beta}}{j!}
		\Big(  1 + \frac{j^2 \log \log t}{\log t}
		+ \frac{j(1+\beta)}{\log t} + O\big( \big( 
		\frac{\log \log t}{\log t}
		\big)^2 \big)
		\Big) \Big)
		\nonumber \\
		~~~~~~~~~~~~~~~~~~~~~~~ 
		~~~~~~~~~~~~~~~~~~~~~~~ 
		~~~~~~~~~~~~~~~~~~~~~~~ 
		+ O((\log t)^{1-d})
		\nonumber \\
		= \begin{cases}
			e^{- \tau e^{-\beta}/j!}
			\left( \exp \big( - \frac{\tau e^{-\beta} j^2 \log \log t}{j!
			\log t} \big)
			+ O\big( (\log t)^{-1} \big)
			\right)
			~~~~~~~~~~~~~~~~~ ~~~ {\rm if}~d=2
			\\
			e^{- \tau e^{-\beta}/j!}
			\left( \exp \big(- \frac{\tau e^{-\beta} j^2 
			\log \log t}{j!
		\log t}
			- \frac{\tau e^{-\beta} j(1+\beta)}{j! \log t} \big)
		+ O\big( \big(\frac{\log \log t}{\log t} \big)^2 \big)
		\right)
			~~~{\rm if}~d \geq 3,
		\end{cases}
		\label{e:Bdetail}
	\end{align}
	yielding 
	(\ref{0114a}).
	Similarly, given also
	$m: \N \to \N$ satisfying $\tau_n:= m(n)/n \to \tau $ as
	$n \to \infty $, 
 by Lemma	\ref{l:Rmeta} and \eqref{e:power}
 we have
		as $n \to \infty$ that
		\begin{align}
			& \Pr[R_{n,m(n),k}(B) \leq r_n ]  
			\nonumber  \\
			 & = \begin{cases}
			e^{- \tau_n e^{-\beta}/j!}
		\exp \Big( - \frac{\tau_n e^{-\beta} j^2 \log \log n}{j!
		\log n}
			\Big)
			+ O\big( (\log n)^{-1} \big)
			~~~~~~~~~~~~~~~~~~ ~~~ {\rm if}~d=2
			\\
			e^{- \tau_n e^{-\beta}/j!}
		\exp \Big( - \frac{\tau_n e^{-\beta} j^2 \log \log n}{j!
		\log n}
		- \frac{\tau_n e^{-\beta} j(1+\beta)}{j! \log n}
		\Big)
		+ O\big( \big(\frac{\log \log n}{\log n} \big)^2 \big)
			~~~{\rm if}~d \geq 3,
		\end{cases}
			\label{e:BBdet}
	\end{align}
	and
	\eqref{0829a}
	follows.
\end{proof}

Under assumption A1 or A2 (with $B=A$),
%(instead of A3),
it takes more work than
in the preceding proof to determine
$r_t$ such that $\gamma_t(A)$
%$t\E[|V_{t,r_t,k}|] $
tends to a finite limit.
The right choice turns out to be as follows.
Let $\beta \in \R$ and let $r_t=r_t(\beta)\ge 0$ be given by
\begin{equation}
	f_0 t\theta_d r_t^d = 
	\max \Big( \big(2-
	%\frac{2}{d} 
	2/d \big) \log t + 
	\big( % \frac{2}{d} +
	2k -4 + 2/d \big)
	J(d,k)
	\log\log t + \beta, 0 \Big),
	%~~~ t\ge t_0,
	\label{e:rboth}
\end{equation}
where $f_0 := |A|^{-1}$ and $J(d,k) :=
	\1_{\{d\ge 3 \text{ or } k\ge 2\}}$. 
%where $t_0$ is chosen sufficiently large that the
%right hand side of (\ref{e:rboth}) is positive for
%all $t \geq t_0$.
We show in  the next subsection that this choice of $r_t$ works. 

\subsection{Convergence of $t \EE[|V_{t,r_t,k}|]$}
	Recall $p_{t}(x)$  at \eqref{e:defptpit}.  By Fubini's theorem,
	as at \eqref{e:defgamma}, we have
\begin{align}\label{e:E[F_n]}
\EE[|V_{t,r_t,k}|]= \int_A p_{t}(x) dx.
\end{align} 
	Recalling the notation $\partial A^{(r)}$ and
	$A^{(-r)}$ from the
	start of Section \ref{s:prep}, we refer to
	 the region $ A^{(-r_t)}$  
	as the {\em bulk}, and the region
$\mathsf{Mo}_{t} := \partial A^{(r_t)}$
%A \setminus A^{(-r_t)}$ 
as the {\em moat}.

\begin{prop}[Convergence of the expectation when $d \geq 3$]
	\label{p:average3d+}
	Suppose Assumption A1 applies with $d \geq 3$.
Fix $\beta \in\R$ and let
	$r_t := r_t(\beta), V_{t,r,k} $ and $ c_{d,k}$
	be as given in \eqref{e:rboth},
		\eqref{e:Vdef} and \eqref{e:defcdk}.
	Then as $t \to \infty$,
\begin{align}
	t \EE[|V_{t,r_t,k}|] & =  e^{-\beta/2} c_{d,k} f_0^{-1/d} |\partial A|
	\Big(1 + \frac{ (k-2 + 1/d)^2  
	\log \log t}{(1-1/d)\log t}
	\nonumber \\ 
	& ~~~~~~~~~~~~~  + \frac{ (k-2 + 1/d)  \beta + 4k -4}{(2-2/d)
	\log t} 
	\Big)
	+ O\big( (\log t)^{\eps -2} \big). 
	\label{e:EFconv}
\end{align}
\end{prop}
	%the case of $d=2$, we have
%$	\EE[|V_{t,r_t,k} |] =  \int_A  p_{t}(x) dx.$
%\end{align}
To prove this,
	we shall
 investigate separately the contributions to the integral in
 the right hand side of
	\eqref{e:E[F_n]}
 from the the different regions $A^{(-r_t)}$ and 
 $\mathsf{Mo}_t$ 
	of the set $A$
	(it turns out  that when $d \geq 3$, the main contribution
	always comes from the moat regardless of $k$.) To avoid repeating ourselves later on, we
	shall deal with 
 $\mathsf{Mo}_t$ in a manner that covers the case $d=2$ as well.

 % {\em Step 1: decomposition of the range of integration.}
%We define the {\em moat} 
%	$\mathsf{Mo}_t := A \cap (\partial A)^{(r_t)}$ 
	%by
%$$
%\mathsf{Mo}_t := \{(u,z):u \in U_t, z \leq \phi_t(x),
%\dist((u,z),\Gamma_t) \leq r_t\}.
%$$
%	and the {\em bulk} $A^{(-r_t)}$.

%{\em Step 2: Contribution of the moat.}
 \begin{lemm}[Contribution of the moat]
	 \label{l:moat}
	Suppose Assumption A1 applies with $d \geq 2$.
Fix $\beta \in\R$ and let
	$r_t := r_t(\beta)$
	% , V_{t,r,k} $ and $ c_{d,k}$
	be  given by \eqref{e:rboth}.
	% \eqref{e:Vdef} and \eqref{e:defcdk}.
Set
	 %$J'(d,k) := {\bf 1}_{\{(d,k)=(2,1)\}}$ and 
	 $J'(d,k) :=1 - J(d,k)$.
		Then
\begin{align}
	t  \int_{\mathsf{Mo}_{t}} p_{t}(x) dx = &
	  c_{d,k} f_0^{-1/d}e^{-\beta/2} |\partial A|
	(\log t)^{-\frac12 J'(d,k)}
	 \Big(1+
	 %(\log t)^{-1}\big(
	 \frac{(k-2+1/d)^2 J(d,k) \log \log t}{(1-1/d) \log t}
	 \nonumber \\
	 &
	 ~~~~~ + \frac{ (k-2+ 1/d) \beta + \ 4k -4}{(2-2/d) \log t}
	 + O\big( (\log t)^{\eps-2} \big)
	 \Big)
	 .
	%+O\big(\big(
	%\frac{\log \log n}{\log n}\big)^2\big) \right)
	\label{e:prisint}
\end{align}
 \end{lemm}
 \begin{proof}
Given $t >0$,  $x \in A$ set
 $\mu_t(x):= t f_0 |B_{r_t}(x) \cap A|$ and
% $x \in \mathsf{Mo}_{t}$ 
	 let $a(x) := \dist(x,\partial A)$. Then by \eqref{e:defptpit},
	 $p_t(x) = \Pr[Z_{\mu_t} <k]$, where $Z_u \sim$ Poisson$(u)$.
	 Also for $t$ large we have $\mu_t(x) \geq 1$ for all $x \in A$.
	 %as discussed below.
	 Hence, similarly to \eqref{e:fromPo} we have uniformly over $x \in A$ that
	 \begin{align}
		 p_{t}(x) & =
		 %\sum_{j=0}^{k-1} e^{-\mu_t(x)} \mu_t(x)^j/j! 
		 %\nonumber
		 %\\
		 %& =
		 ((k-1)!)^{-1}  e^{- \mu_{t}(x) } \mu_{t}(x)^{k-1} 
		 (1+ ((k-1)/ \mu_t(x)) + O((\mu_t(x))^{-2})).
	 \label{e:PoMass}
	 \end{align}

 For $a \in (0,1]$ set
%$h(a) := |B_1(o) \cap ([0,a] \times \R^{d-1})|$,
%as in Section \ref{ss:keylem}, and set
%$\Lambda_{t,a}:= tf_0  r_t^2((\pi/2) + h(a))$.  with
$\Lambda_{t,a}:= tf_0  r_t^d(\frac12 \theta_d + h(a))$,
	 with $h(\cdot)$ defined at \eqref{e:defh}.
%a higher-dimensional analogue of \eqref{eq_pnj_rec}. 
	 By Lemma \ref{l:bdyvol} and \eqref{e:rboth},
%and a simple coupling, as well as
%along with \eqref{e:defptpit},
%for all $x \in \mathsf{Mo}_{t}$,
%setting 
we have that 
$$
\sup_{x \in \mathsf{Mo}_t} |\mu_t(x) - \Lambda_{t,a(x)/r_t} |
= O(t r_t^{d+1}) = O \Big( (\log t)^{(d+1)/d} t^{-1/d} \Big)
$$
where we have used also 
the fact that $tr_t^d = \Theta (\log t)$.
	 Also 
$\Lambda_{t,a}=\Theta(\log t)$ uniformly over $a\in[0,1]$.
Hence,
for each $x\in \mathsf{Mo}_{t}$,
	 by \eqref{e:PoMass}
we have
\begin{align*}
	p_t(x) =  ((k-1)!)^{-1}  e^{- \Lambda_{t,a(x)/r_t} } 
	\Lambda_{t,a(x)/r_t}^{k-1} (1+( (k-1)/\Lambda_{t,a(x)/r_t}) +
	O((\log t)^{-2})),
	%\label{e:pnk}
\end{align*}
%{\bf [may be able to improve the error term} when $k=1$: seems it should
%	 then be $O(t^{\eps-1/d})$ instead of $O((\log t)^{-2})$]
where the constant in the $O$ term is independent of $x$.
Then by Proposition \ref{thm:cov}, 
%(\ref{e:pnk}),
\begin{align*}
	t \int_{\mathsf{Mo}_{t}} p_{t}(x)dx
	= \left( \frac{t |\partial A|(1+ O(r_t))}{(k-1)!} \right) r_t
	\int_0^1 e^{-\Lambda_{t,a}} \Lambda_{t,a}^{k-1}
	\Big(1+ \frac{k-1}{\Lambda_{t,a}} + O\big((\log t)^{-2}\big) \Big)da.
\end{align*}
By \eqref{e:rboth}, 
for $t$ large
$e^{-tf_0\theta_d r_t^d/2}= t^{-(1-1/d)} (\log t)^{(2-k-1/d)J(d,k)}
e^{-\beta/2}$, and setting $s = tf_0r_t^d$ we have
$\Lambda_{t,a}= s(\frac12 \theta_d + h(a))$, so that
\begin{align*}
	t \int_{\mathsf{Mo}_{t}} p_{t}(x)dx
	& = \Big( \frac{ t |\partial A| r_t}{(k-1)!} \Big)
	t^{-(1-1/d)} 
	(\log t)^{(2-k-1/d)J(d,k)} e^{-\beta/2} s^{k-1} \\
	& \times \int_0^1 e^{-sh(a)} \left( \frac{\theta_d}{2} + h(a)
	\right)^{k-1} \left(1+ \frac{k-1}{s(\frac{\theta_d}{2}+ h(a))}
	+ O((\log t)^{-2}) \right) da.
\end{align*}
Hence by Lemma \ref{lemexp} with $\alpha_0= \theta_d/2$ and $j=k-1$, 
given $\eps >0$ we have
\begin{align*}
	& t \int_{\mathsf{Mo}_{t}} p_{t}(x)dx
	= \Big( \frac{ t |\partial A| r_t}{(k-1)!} \Big)
	t^{-(1-1/d)} 
	(\log t)^{(2-k-1/d)J(d,k)} e^{-\beta/2} s^{k-1} \\
	& ~~~~~~~~~~~~~~~~~~~~~~~~~~~~~~~  
	\times  \frac{(\theta_d/2)^{k-1}}{\theta_{d-1}s}
	\left( 1 + 
	4(k-1)\theta_d^{-1} 
	s^{-1}
	+
	O(s^{\eps -2}) 
	\right) \\
	& = \frac{ f_0^{-1/d} \theta_d^{k-1} |\partial A|
	e^{-\beta/2}
	}{(k-1)!2^{k-1} \theta_{d-1}} \left( \frac{s}{\log t} \right)^{k-2+1/d}
	(\log t)^{-\frac12 J'(d,k)}
	( 1 + 4(k-1)\theta_d^{-1}  s^{-1} + O(s^{\eps -2})).
\end{align*}
By \eqref{e:rboth}, for $t$ large we have 
$$
s= \frac{(2-2/d) \log t}{\theta_d} \left( 1 + 
\frac{(k -2 + 1/d) J(d,k) \log \log t + \beta /2 }{(1-1/d)\log t}
\right).
$$
Therefore
\begin{align}
	t  \int_{\mathsf{Mo}_{t}} p_{t}(x) dx =
	 \frac{ f_0^{-1/d} \theta_d^{k-1} |\partial A|
	e^{-\beta/2}
	}{(k-1)!2^{k-1} \theta_{d-1}} \left( \frac{2-2/d}{\theta_d} \right)^{k-2+1/d}
	(\log t)^{-\frac12 J'(d,k)}
	~~~~~~~~~~~~~~~
	\nonumber \\
	%~~~~~~~~~~~~~~~~~~~
	\times \left( 1 + \frac{(k-2+1/d)  ((k-2 + 1/d)J(d,k) \log \log t
	+  \beta/2)}{ (1- 1/d)\log t} 
	+O\big(\big(
	\frac{\log \log t}{\log t}\big)^2\big) \right)
	\nonumber \\
	 \times \left( 1 + \frac{4k-4 }{(2-2/d) \log t } +
	 O \big( (\log t)^{\eps-2}   \big)\right),
	 \nonumber
\end{align}
and therefore by the definition \eqref{e:defcdk} of $c_{d,k}$,
	\eqref{e:prisint} holds.
	\end{proof}
%Hence
%
%Combining %\eqref{e:int_da}-
%\eqref{e:prisint} and
%\eqref{e:error_set}, using
%\eqref{e:error_corn}
%and setting 
%$M_t := \cup_{i=1}^{m(t)} \mathsf{Pri}_{t,i} $, leads to
%\begin{align}
%	%t \EE[|V_{t,r_t,k} \cap \mathsf{Mo}_t|] & =
%	t \int_{\mathsf{Mo}_t} p_t(x) dx & =
%	t  \int_{M_t} p_{t}(x) dx + O((\log t)^{2\kappa +2/d}
%	t^{(1-1/d)(1-2 \kappa) + \alpha - 1/d}) 
%	\nonumber \\
%	& =
%	e^{-\beta/2} c_{d,k} f_0^{-1/d} |\Gamma| 
%	\Big(1+  \frac{(k-2+ 1/d)^2 \log \log t}{(1-1/d) \log t} 
%	\nonumber \\
%	& ~~~~~~~~~~ + \frac{ (k -2 + 1/d)  \beta + 4k -4}{(2-2/d) \log t}
%	\Big)
%	+ O\big(   (\log t)^{\eps -2}  \big). 
%	\label{0131c}
%\end{align}

\begin{proof}[Proof of Proposition \ref{p:average3d+}]
	To deal with the bulk, we use \eqref{e:PoMass},
	noting that for $x \in A^{(-r_t(\beta))}$ we have
	$\mu_t(x) = t f_0 \theta_d r_t^d$. Hence for such $x$ we have
	$$
	p_t(x) = O((t r_t^d)^{k-1} e^{-t f_0 \theta_d r_t^d}) =
	O(
	( \log t)^{k-1}
	t^{-(2-2/d)} (\log t)^{4- 2k-2/d}
	),
	$$
%	note that for $t$ large and each $x\in A^{(-r_t(\beta))}$, 
%by \eqref{e:defptpit} and \eqref{e:rboth} we have
%\begin{align*}
%	p_{t}(x) & = \sum_{j=0}^{k-1} ((tf_0 \theta_d r_t^d)^j/j!) 
%e^{-t f_0 \theta_d r_t^d}
%\\
%	& \leq k t^{-(2-2/d)} (\log t)^{4- 2k-2/d} e^{-\beta} 
%	(2 \log t)^{k-1},
%\end{align*}
	where the constant  in the $O$ term does not depend on $x$,
so that
%by (\ref{e:E[F_nS]}),
\begin{align}
t
	%\EE[|V_{t,r_t,k} \cap A^{(-r_t(\beta))}|]
	\int_{A^{(-r_t)}} p_t(x)dx
	%\le |A|
	= O(	(\log t)^{2} t^{-1+ 2/d }).
	\label{e:bulk3+}
\end{align}

%{\em Step 4: putting together the contributions.}
 Using
 (\ref{e:prisint}), (\ref{e:bulk3+}) and  
 %(\ref{e:defsigA}),
 \eqref{e:E[F_n]},
 we obtain (\ref{e:EFconv}) for $d \geq 3$ as required.
\end{proof}

%\subsection{Convergence of $t\EE[|V_{t,r_t,k}|]$ for $d=2$}

%We now demonstrate convergence of $t \E[|V_{t,r_t,k}|] $ when $d=2$.

\begin{prop}[Convergence of the expectation
	%$\EE[t|V_{t,r_t,k}|]$
	when $d=2$]	
	\label{p:average2d}
	Suppose $d=2$ and A1 or A2 applies.
	Fix $\beta \in\R$,
	and let $r_t, V_{t,r,k}$ be as given in 
	\eqref{e:rboth}
	and
	\eqref{e:Vdef}.
	Then as $t \to \infty$,
	\begin{align}
		t\EE[|V_{t,r_t,k}|] =  
		\begin{cases}
			|A| e^{-\beta} + |\partial A|e^{-\beta/2} 
			\frac{\sqrt{\pi/f_0}}{2} (\frac{1}{\sqrt{\log t}} + O((\log t)^{-3/2}) ) & \mbox{ if } k= 1 \\
			|A| e^{-\beta} + |\partial A|e^{-\beta/2} 
			\frac{\sqrt{\pi/f_0}}{4}
			\Big( 1 +
			 \frac{\log\log t}{2 \log t} 
			 \Big)
			 + \frac{|A|e^{-\beta}\log \log t}{\log t}
			 \\
			~~~~~~~~~~~~~~~~~~~~~~~~~~~~
			~~~~~~~~~~~+ O((\log t)^{-1}) )
			 & \mbox{ if } k=2 
			 \\
			|\partial A|e^{-\beta/2} \frac{\sqrt{\pi/f_0}}{(k-1)!2^k}
			\left(1+\frac{(2k-3)^2\log\log t 
			%+ (2k-3)\beta + \frac{\pi}{2} + 2k -2
			}{2 \log t}
			\right)
			+ 
			%O((\frac{\log \log  t}{\log t})^{2}) \right)
			O\big(\frac{1}{\log t}\big) 
			& \mbox{ if } k\ge 3.
		\end{cases}
%		e^{-\be}|A|? +  ?O((\log n)^{-1/2})?.
		\label{e:EFlim}
	\end{align}
\end{prop}

\begin{proof} 

	{\bf Case 1: $A$ has a $C^{1,1}$ boundary and $\overline{A^o}=A$.}
%
%
% \begin{align}\label{e:def_A_shrink}
% \An^1 &= \{x\in A: \dist(x,\partial \An) \ge   r_t(\beta)  \}; 
%\end{align} 
%
%({\bf Maybe use $\Delta$ instead of $B$ here?)}. 
%The letter 
%$B$ refers to boundary and should not be confused with the balls which are always denoted by $B_r(x)$ or $B(x,r)$.  
%As such,
%	Thus $A$ is partitioned into the bulk and the moat.  \\
%
%\noindent{\it Step 1:
	We first estimate the contribution to \eqref{e:E[F_n]}
from the bulk. By
	%(\ref{e:defptpit}), 
	\eqref{e:PoMass},
	for $x\in A^{(-r_t)}$ we have 
$$
p_{t}(x) = \frac{1}{(k-1)!} e^{ - tf_0\pi r_t^2}
(tf_0\pi r_t^2)^{k-1} \Big(1+ \frac{k-1}{tf_0\pi r_t^2} + O(\frac{1}{(tr_t^2)^2}) \Big),
$$
 where the O term is $0$ when $k=1$ or $k=2$.
 Also by \eqref{e:rboth}, for $t$ large
 we have $e^{-tf_0 \pi r_t^2}= e^{-\beta}t^{-1}
	(\log t)^{(3-2k)\1_{\{k \geq 2\}}}$. Hence for $k \ge 2$ and $x \in A^{(-r_t)}$
 we have
 \begin{align*}
 p_{t}(x) = (t^{-1} e^{-\beta}/(k-1)!) (\log t)^{3-2k} (\log t +
 (2k-3)\log \log t + \beta)^{k-1} \\
	 \times (1+ (k-1)(\log t)^{-1}(1+ O((\log \log t)/\log t))),
 \end{align*}
	while if $k=1$ and $x \in A^{(-r_t)}$ then $p_t(x) = e^{-\beta}t^{-1}$.
	Hence, since $|A^{(-r_t)}| = |A_t| + O(r_t)$,
\begin{align} 
	\int_{A^{(-r_t)}} p_{t}(x) tdx
	=\begin{cases}
		e^{-\beta} |A| (1+ O(r_t)) & \mbox{ if }  k=1, \\
		e^{-\beta} |A| \big(1+\frac{\log\log t}{\log t} + 
		\frac{\beta + 1}{\log t}  + O( \frac{\log \log t}{(\log t)^2})
		\big ) & \mbox{ if } k=2, \\
		\frac{e^{-\beta}}{(k-1)!(\log t)^{k-2}} + O(\frac{\log \log t}{
			(\log t)^{k-1}})& \mbox{ if } k\ge 3.
	\end{cases}
	\label{e:2dbulk}
\end{align}

	For the contribution from the moat, we use Lemma \ref{l:moat}.
	Note that $c_{2,k} = \frac{\pi^{1/2}}{(k-1)!2^k}$.
	By
	\eqref{e:prisint},
\begin{align}
	t \int_{\mathsf{Mo}_t}p_{t}(x) dx  = 
%	~~~~~~~~~~~~~~~~~~~~~~~~~~~~~~~~~~~~~~
%	~~~~~~~~~~~~~~~~~~~~~~~~~~~~~~~~~~~~~~
%	\nonumber
%	\\
	\begin{cases}
		 \frac12 |\partial A|  e^{-\beta/2} \sqrt{\pi/f_0} ((\log t)^{-1/2} + 
		 O((\log t)^{-3/2})) & {\rm if}~ k=1, \\
		 |\partial A|  e^{-\beta/2} \frac{\sqrt{\pi/f_0}}{(k-1)!2^k}
		 %\Big(1 + \frac{(2k-3)^2\log\log t
		 %+ (2k-3)\beta + \frac{\pi}{4} + 2k -2}{2 \log t} + 
		 %O\big(\big(\frac{\log \log t}{\log t}\big)^2\big)\Big) 
		 \Big(1 + \frac{(2k-3)^2\log\log t
		 }{2 \log t } \Big) + 
		 O\big(\frac{1}{\log t} \big) 
		 & {\rm if}~ k\ge 2.
	\end{cases}
	%\le \frac{c}{\sqrt{\log n}}.
	\label{e:2dmoat}
\end{align}
Combining this with (\ref{e:2dbulk}),
%and (\ref{e:2dmoat}), 
and
using (\ref{e:E[F_n]}), yields (\ref{e:EFlim}) in Case 1. \\
%The claim follows by combining all the estimates. 

{\bf Case 2: $A$ is polygonal.} In this case,
the contribution of the bulk $A^{(-r_t)}$ to the integral on the right
hand side of (\ref{e:E[F_n]}) can be dealt with just as in Case 1; that is,
(\ref{e:2dbulk})
remains valid in this case.

Let $|\Phi_0(A)|$ denote the number of corners of $A$
and enumerate the edges and corners of $A$ in some arbitrary order.
For $1 \leq i \leq |\Phi_0(A)|$,
let $\mathsf{Rec}_{t,i}$
denote a rectangular region in $A$  having as its base part
of the $i$th edge of $A$. 
We take each rectangle to have width $r_t$ and
each end of each rectangle to  be distant $K r_t$ 
from the corner of $A$ at the corresponding end of the corresponding
edge of $A$, with $K$ chosen large enough so that $K> 3/|\sin(\alpha/2)|$
for each angle $\alpha $ of the polygon $A$, as shown in 
Figure \ref{f:polyrecs}.
This choice of $K$ ensures that  the  rectangular
regions are pairwise disjoint.

	\begin{figure}[!h]

\center
\includegraphics[width=5cm]{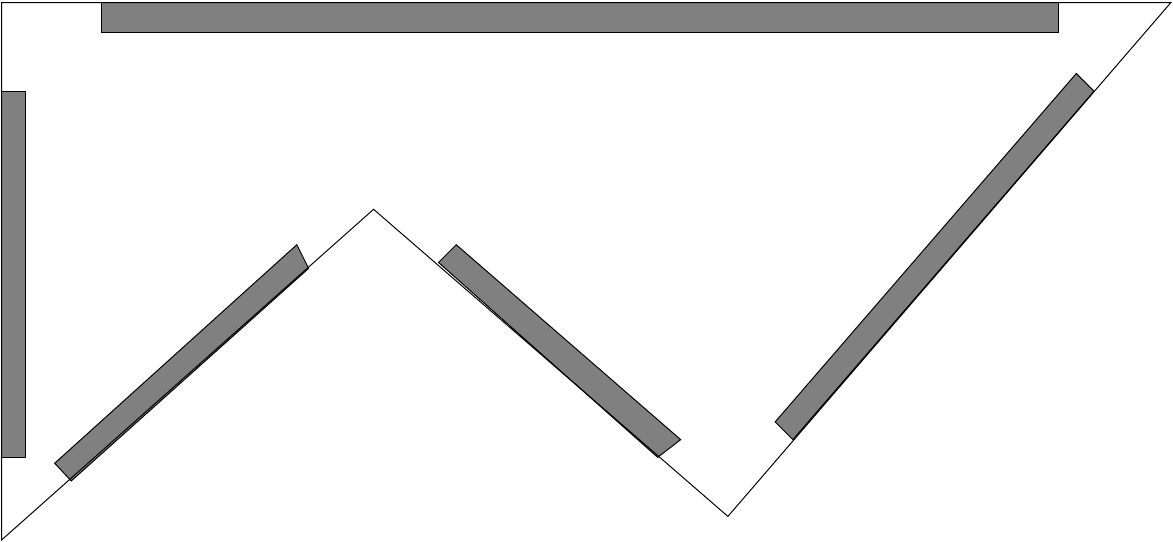}
%\includegraphics[width=1.2\textwidth]{pic1.pdf}
		%\vspace{-2cm}
		\caption{\label{f:polyrecs}Illustration showing
		the rectangles $\mathsf{Rec}_{t,i}$ (shaded) in the proof of
		Proposition \ref{p:average2d}, Case 2.
        }
\end{figure}

We define the corner region $\mathsf{Cor}_{t} := \mathsf{Mo}_t \setminus \cup_{i=1}^{|\Phi_0(A|} 
\mathsf{Rec}_{t,i}$. 
%Then $\mathsf{Cor}_t$ is contained the union
%of corner regions $\mathsf{Cor}_{t,i}$ similarly to before; now
Let 
$\mathsf{Cor}_{t,i}$ denote the intersection of $A$ with  the disk of
radius $(K+1)r_t$ centred on the $i$th corner of $A$. 
Then $\mathsf{Cor}_t \subset \cup_{i=1}^{|\Phi_0(A)|}\mathsf{Cor}_{t,i}$.

Then there exists 
$\kappa  >0 $ (depending on the sharpest angle of $A$), such that
for all  large $t$  and any $i \leq |\Phi_0(A)|$,
$x\in \mathsf{Cor}_{t,i}$, we have
$p_t(x)\le e^{- tf_0 \pi r_t^2 \kappa}$ and
hence by \eqref{e:rboth}, $p_t(x) \leq t^{-\kappa}$ (note that the
coefficient of $\log \log t$ in \eqref{e:rboth} is positive in this case).
Each $\mathsf{Cor}_{t,i}$ has area at most $\pi ((K+1)r_t)^2$, and hence
\begin{align}
	t f_0 \int_{\sf{Cor}_{t,i}} p_t(x)dx 
	%\le  c t r_t^2 t^{-\kappa}
	= O(    t r_t^2 t^{-\kappa} )
	%e^{- tf_0 \pi r_t^2 \kappa}
	= O( (\log t) t^{-\kappa}),
	\label{e:Corest}
\end{align}
leading to a total contribution 
from the corners to (\ref{e:E[F_n]})  
of  $O((\log t)t^{-\kappa })$.
Since there are a fixed number of corner regions,
so the total contribution of these regions to the integral on
the right hand side of 
(\ref{e:E[F_n]})  is $O((\log t)t^{-\kappa })$.

By the proof of Lemma \ref{l:moat}
we see that the total contribution of the
rectangles $\mathsf{Rec}_{t,i}, 1 \leq i \leq |\Phi_0(A)|$,
to the right hand side of 
(\ref{e:E[F_n]})  is the same as \eqref{e:2dmoat}.
There is an extra multiplicative error term of $O(r_t)$ due to 
the total length of the rectangles being less than the
perimeter of $A$, but this error term is dominated by
the error terms already included in \eqref{e:2dmoat}.

Putting together these estimates yields (\ref{e:EFlim}) in Case 2.
\end{proof}

\subsection{Proof of theorems \ref{thsmoothgen} and \ref{t:dhi}}

\begin{proof}[Proof of Theorem \ref{thsmoothgen}]
	Recall $f_0 := 1/|A|$.
	Suppose $(r_t)_{t>0}$ satisfies
	the case $d=2,k=1$ of \eqref{e:rboth}, so that
	$t \pi f_0 r_t^2 - \log t = \beta$
	for all large enough $t$.
	Then by
	\eqref{e:defgamma},
	Proposition \ref{p:average2d}
	and
	\eqref{e:defsigA}, 
	\begin{align*}
		\gamma_t(A):=
		t f_0 \E[|V_{t,r_t}|] =  e^{-\beta} +
		(\sigma_A e^{-\beta/2} \pi^{1/2}/2) (\log t)^{-1/2}
		+ O((\log t)^{-3/2}).
	\end{align*}
	Hence by Lemma \ref{l:Poapprox},
	\begin{align*}
		\Pr[R'_{t,\tau t} \leq r_t]
		= \exp \big( - \tau e^{-\beta} -  \tau 
		(\sigma_Ae^{-\beta/2} \pi^{1/2}/2)
		(\log t)^{-1/2} \big) + O((\log t)^{-1}),
	\end{align*}
	and hence \eqref{eqmain3}. Also by Lemma  \ref{l:Rmeta},
	setting $\tau_n = m(n)/n$ we have
	\begin{align*}
		\Pr[R_{n,m } \leq r_n]
		= \exp \big( - \tau_n e^{-\beta} -  \tau_n 
		(\sigma_Ae^{-\beta/2} \pi^{1/2}/2)
		(\log n)^{-1/2} \big) + O((\log n)^{-1}),
	\end{align*}
	and hence \eqref{0829b}.
\end{proof}

\begin{proof}[Proof of Theorem \ref{t:dhi} for $d= 2$]
	Take $d=2, k \geq 2$. Let $\beta \in \R$ and 
define	$(r_t)_{t > 0}$
	by \eqref{e:rboth},
	so that
	$t \pi f_0 r_t^2 - \log t 
	+(3-2k) \log \log t = \beta$ for $t $ large.
	
	First suppose $k=2$.  By \eqref{e:defgamma}
	and Proposition \ref{p:average2d},
	\begin{align*}
		\gamma_t(A) =
		t f_0 \E[|V_{t,r_t,k}|]=  e^{-\beta}
		+ \sigma_A e^{-\beta/2} (\pi^{1/2}/4)
		\Big( 1 + \frac{\log \log t}{2 \log t} \Big) 
		+ O((\log t)^{-1}).
	\end{align*}
	Hence by Lemma \ref{l:Poapprox},
	\begin{align*}
		\Pr[R'_{t,\tau t,k} \leq r_t] =
		\exp \Big(- \tau
		 e^{-\beta}
		 -\tau \sigma_A e^{-\beta/2} (\pi^{1/2}/4)
		\Big(1 + \frac{\log \log t}{2 \log t}\Big) 
		\Big) + O((\log t)^{-1}) ,
	\end{align*}
	and hence \eqref{e:Polim22}.
	Also by Lemma \ref{l:Rmeta}, with $\tau_n= m(n)/n$,
	\begin{align*}
		\Pr[R_{n,m(n),k} \leq r_n] =
		\exp \Big(- \tau_n e^{-\beta}
		 -\tau_n \sigma_A e^{-\beta/2} (\pi^{1/2}/4)
		\Big(1 + \frac{\log \log n}{2 \log n} \Big) 
		\Big)
		+ O((\log n)^{-1}) ,
	\end{align*}
	and hence \eqref{e:lim22}.

	Now suppose $k \geq 3$.  By \eqref{e:defgamma}
	and Proposition \ref{p:average2d},
	\begin{align*}
	%	t f_0 \E[|V_{t,r_t,k}|]=
		\gamma_t(A) = 
		 \frac{\sigma_A e^{-\beta/2} \pi^{1/2}}{(k-1)!2^k}
		\Big(1 + \frac{(2k-3)^2\log \log t}{2 \log t}\Big) 
		+ O((\log t)^{-1}).
	\end{align*}
	Hence by Lemma \ref{l:Poapprox},
	\begin{align*}
		\Pr[R'_{t,\tau t,k} \leq r_t] =
		\exp \Big(- 
		 \frac{\tau \sigma_A e^{-\beta/2} \pi^{1/2}}{
			 (k-1)!2^k}
			 \Big(1 + \frac{(2k-3)^2 \log \log t}{2 \log t}\Big) 
		\Big)
		+ O((\log t)^{-1}),
	\end{align*}
	and hence \eqref{e:Polimgen} for $d=2, k \geq 3$
	(note that $c_{2,k} = (\pi^{1/2}) / ((k-1)!2^k)$ by \eqref{e:defcdk}).
	Also by Lemma \ref{l:Rmeta}, setting $\tau_n := m(n)/n$,  we have
	\begin{align*}
		\Pr[R_{n,m(n),k} \leq r_n] =
		\exp \Big(
		 -\frac{\tau_n \sigma_A e^{-\beta/2} \pi^{1/2}}{
			 (k-1)!2^k}
		\Big(1 + \frac{(2k-3)^2 \log \log n}{2 \log n} \Big) 
		\Big)
		+ O((\log n)^{-1}) 
		,
	\end{align*}
	and hence \eqref{e:limgen} for $d=2, k \geq 3$.
		%Together, these two results give us the asserted results for
		%$d=2, k \geq 2$.
\end{proof}

\begin{proof}[Proof of Theorem \ref{t:dhi} for $d\geq 3$]

	Assume $d \geq 3, k \geq 1$. Let $\beta >0$ and 
	let $(r_t)_{t >0}$ satisfy \eqref{e:rboth},
	so $t f_0 \theta r_t^d = (2-2/d) \log t
	+(2k-4 +2/d) \log \log t + \beta$ for $t$ large.
	Then by Proposition \ref{p:average3d+},
	we have \eqref{e:EFconv}, and hence
\begin{align*}
	\gamma_t(A) =	t f_0 \EE[|V_{t,r_t,k}|] & =   e^{-\beta/2} c_{d,k} \sigma_A
	\Big(1 + \frac{ (k-2 + 1/d)^2  
	\log \log t}{(1-1/d)\log t}   
	\\ 
	& + \frac{ (k-2 + 1/d)  \beta + 4k -4}{(2-2/d)
	\log t} 
	\Big)
	+ O\big( (\log t)^{\eps -2} \big). 
	%O( \frac{\log \log n}{\log n}).
\end{align*}
	Hence by Lemma \ref{l:Poapprox} we have
	\begin{align}
		\Pr[R'_{t,\tau t,k} \leq r_t] 
		& = \exp \Big( - \tau   e^{-\beta/2} c_{d,k} \sigma_A
		\Big(1 + \frac{ (k-2 + 1/d)^2  
		\log \log t}{(1-1/d) \log t}  
		\nonumber
	\\ 
		& + \frac{ (k-2 + 1/d) \beta + 4k -4}{(2-2/d)
		\log t}  
		\Big) \Big)
		+ O\big((\log t)^{\eps -2} \big), 
		\label{e:dkPobigdet}
\end{align}
	which gives us \eqref{e:Polimgen},
	and by Lemma \ref{l:Rmeta}, setting $\tau_n = m(n)/n$
	we have 
	\begin{align}
		\Pr[R_{n,m,k} \leq r_t] 
		& = \exp \Big( - \tau_n  e^{-\beta/2} c_{d,k} \sigma_A
		\Big(1 + \frac{ (k-2 + 1/d)^2 
		\log \log n}{(1-1/d) \log n}
	\nonumber \\ 
		& + \frac{ (k-2 + 1/d)  \beta + 4k -4}{(2-2/d)
		\log n} \Big) \Big)  
		+ O\big( (\log n)^{\eps -2} \big),
		\label{e:dkbigdet}
\end{align}
	giving us \eqref{e:limgen}.
\end{proof}

\section{Simulation results and discussion}
\label{sim-section}

\begin{figure}[p]
\center
\includegraphics[width=0.49\textwidth, trim= 0 10 5 20, clip]{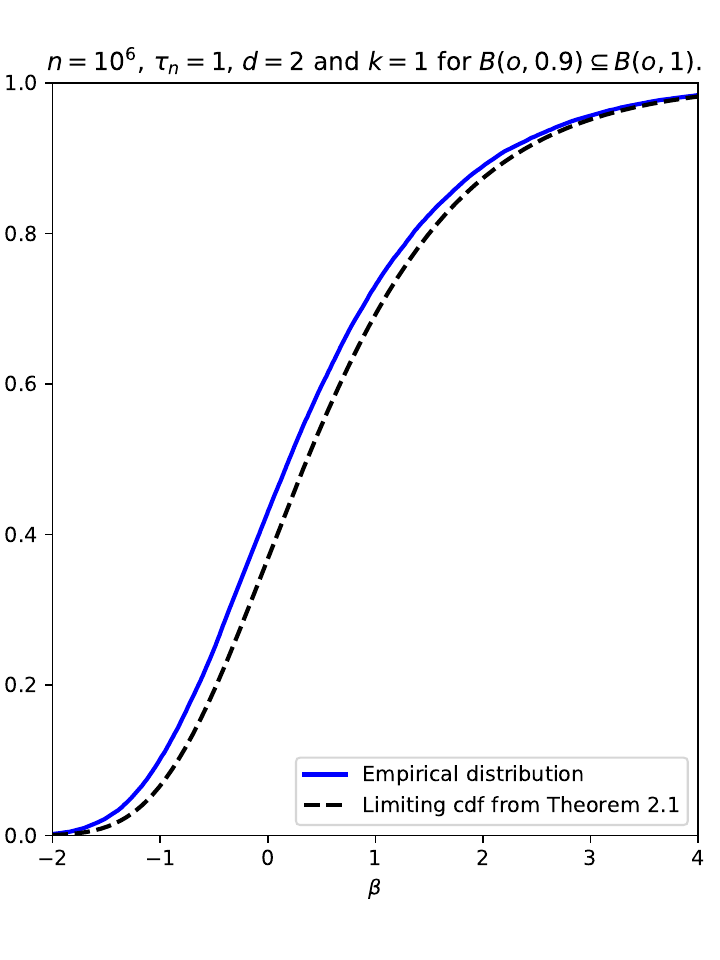}
\includegraphics[width=0.49\textwidth, trim= 0 10 5 20, clip]{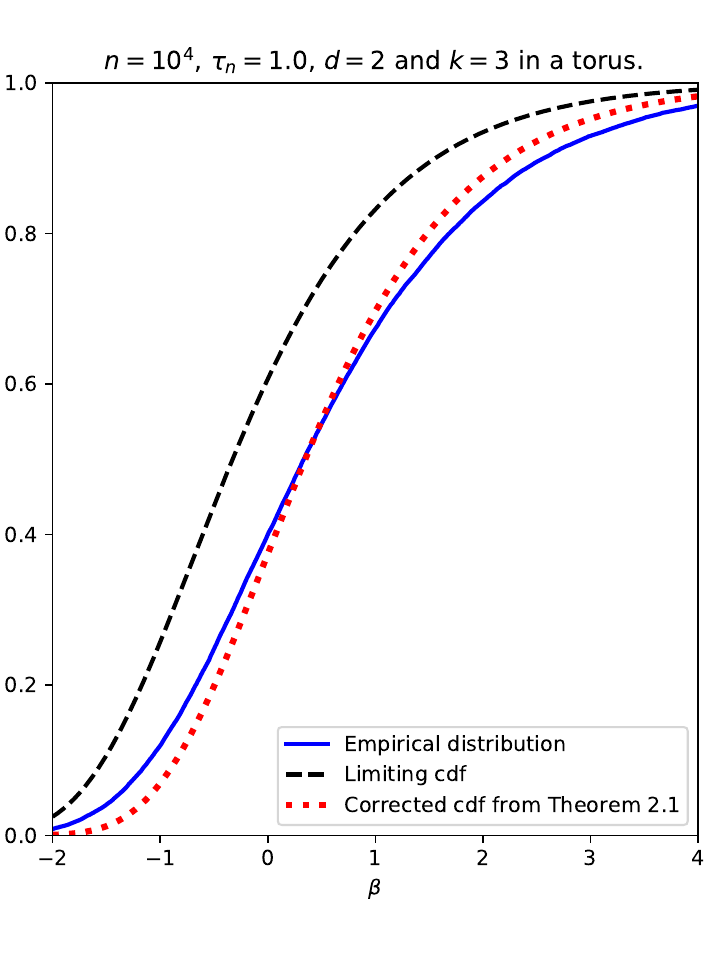}
\includegraphics[width=0.49\textwidth, trim= 0 10 5 20, clip]{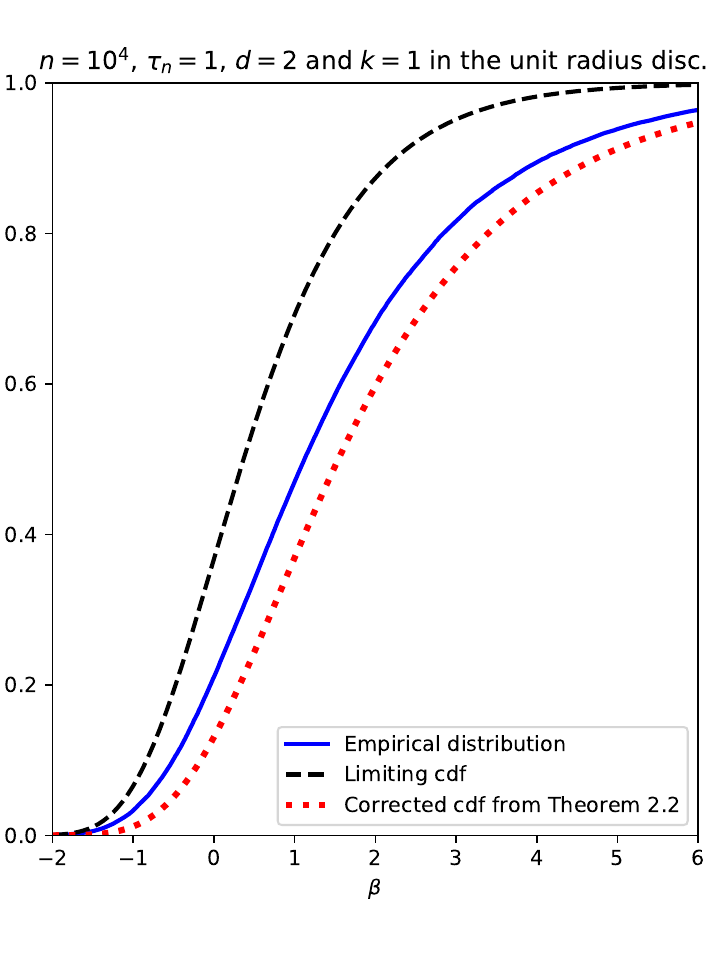}
\includegraphics[width=0.49\textwidth, trim= 0 10 5 20, clip]{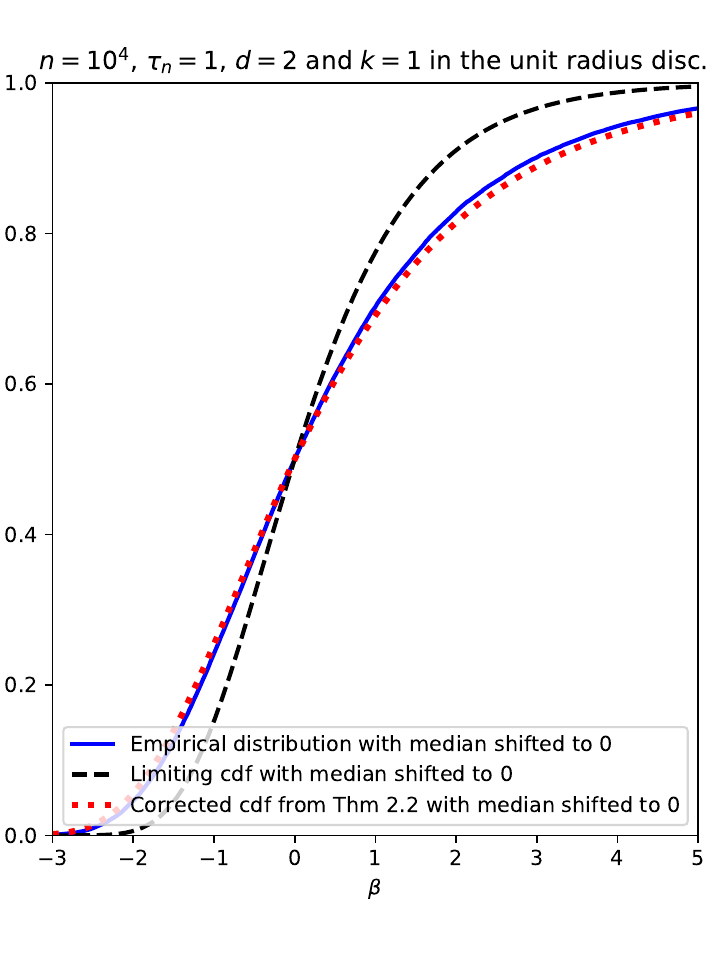}
\caption{\label{simulations1}
The empirical distributions of $n \theta_d f_0 R_{n,m(n),k}^d - c_1 \log n - c_2 \log\log n$
obtained from computer simulations in the settings of Theorems \ref{Hallthm}, \ref{thsmoothgen} and \ref{t:dhi}, plotted on the same axes as the limiting distributions. See Section~\ref{sim-section} for discussion of the simulation results.}
\end{figure}

\begin{figure}[p]
\center
\includegraphics[width=0.49\textwidth, trim= 0 10 5 20, clip]{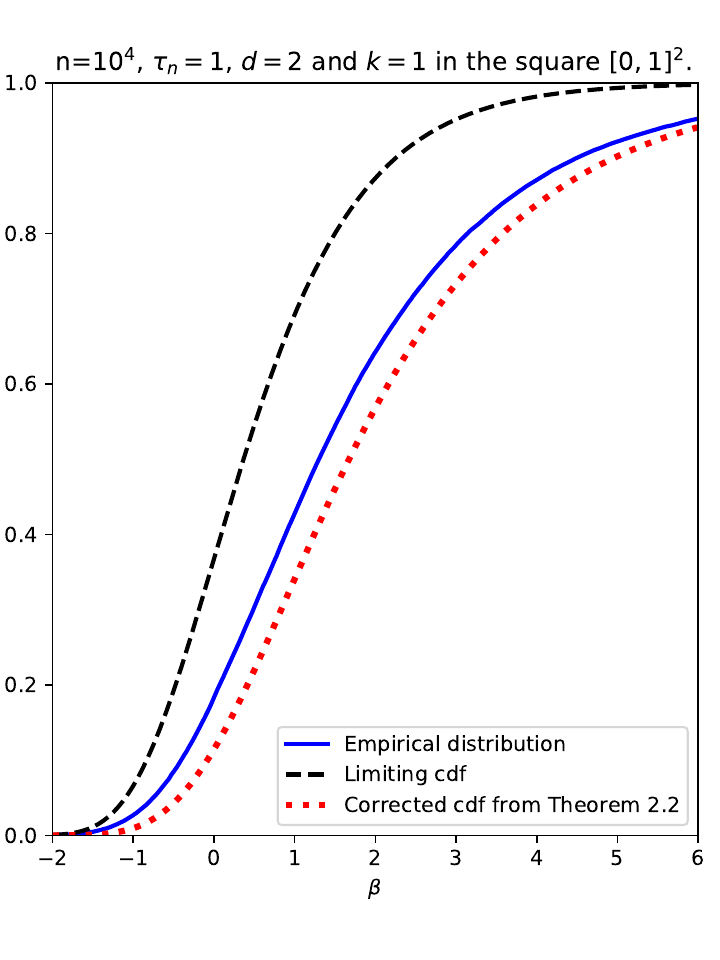}
\includegraphics[width=0.49\textwidth, trim= 0 10 5 20, clip]{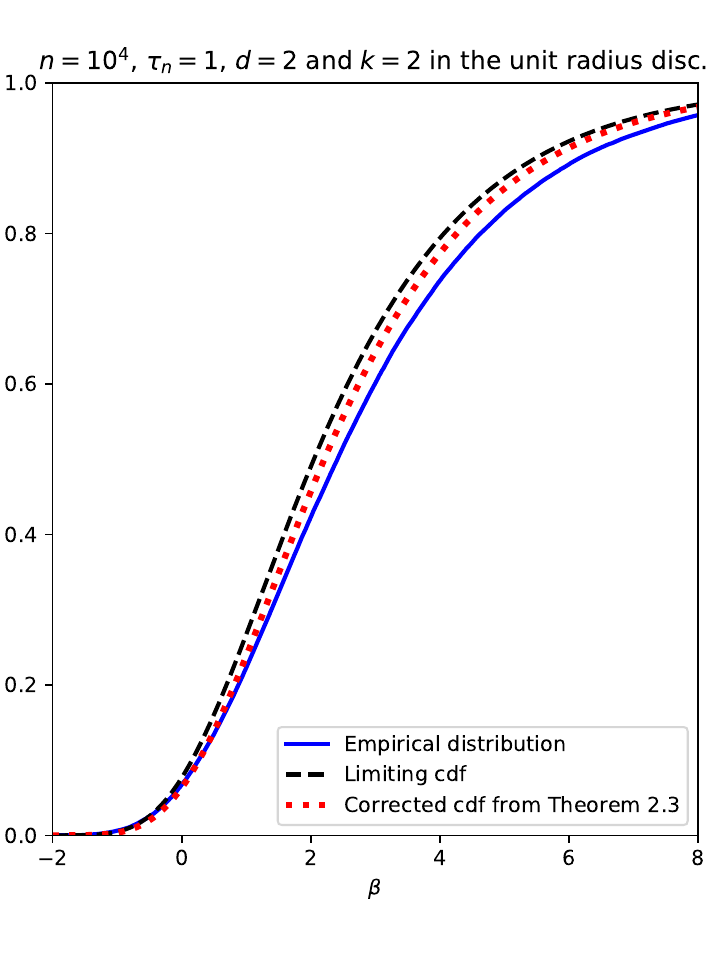}
\includegraphics[width=0.49\textwidth, trim= 0 10 5 20, clip]{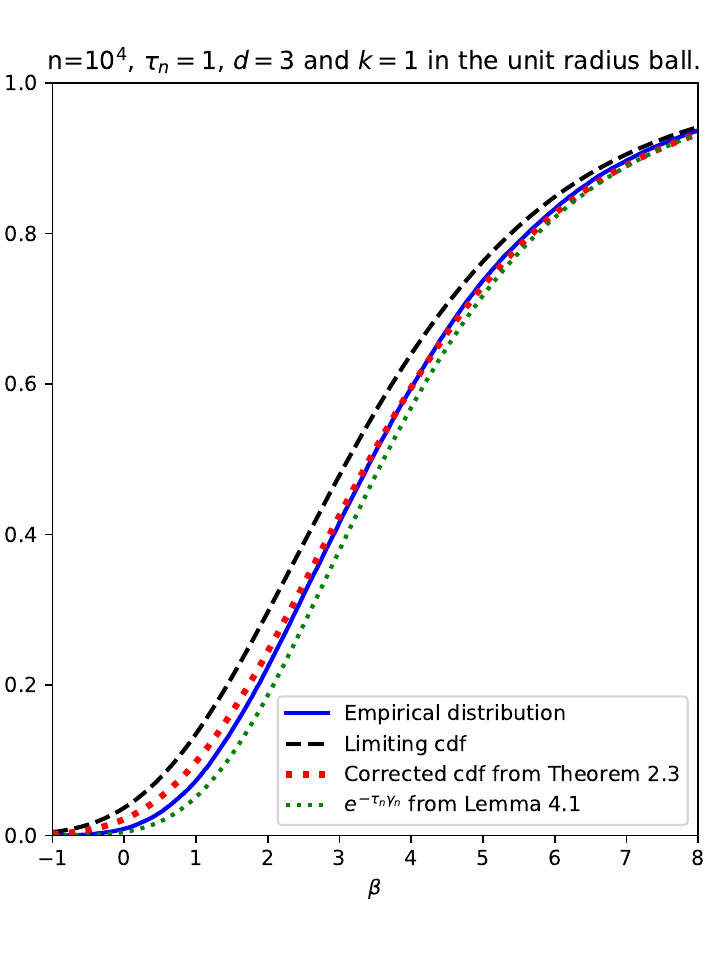}
\includegraphics[width=0.49\textwidth, trim= 0 10 5 20, clip]{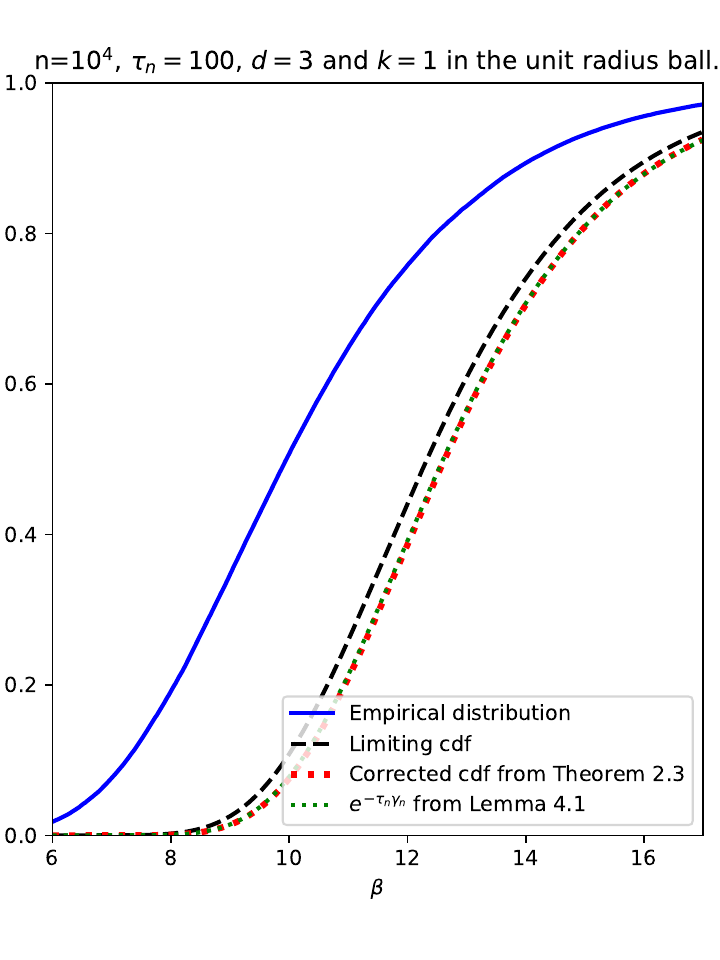}
\caption{\label{simulations2}
The continuation of Figure~\ref{simulations1},
with results from simulations in several more settings.}
\end{figure}

We were able to write computer simulations which
sample from the distribution of $R_{n,m(n),k}$
using a very simple algorithm:
sample $n + m(n)$ independent points
$X_1, \dots, X_n, Y_1, \dots, Y_{m(n)}$.
For each $j \in \{1, \dots, m(n)\}$
let $d_j^{(k)}$ be the Euclidean distance between $Y_j$ and its $k$th-nearest point in $\mathcal{X}_n = \{ X_1, \cdots, X_n \}$.
Then $R_{n,m(n),k} = \max_{j \leq m(n)} d_j^{(k)}$.

In Figures~\ref{simulations1} and \ref{simulations2}, we present the results from simulations of many of the settings for which we have proved limit theorems.
In each of the eight plots, the blue curve is an estimate of the cumulative distribution function of the quantity of the form
$n \theta_d f_0 R_{n,m(n),k}^d - c_1 \log n - c_2 \log\log n$ for which we have obtained weak laws.
These distributions were estimated by sampling several tens of thousands of times from the distribution of $R_{n,m(n),k}$ and plotting the resulting empirical distribution.
The black dashed curves are the corresponding limiting distributions as $n \to \infty$, from Theorems \ref{Hallthm}, \ref{thsmoothgen} and \ref{t:dhi}.
The red dotted curves are the corresponding ``corrected'' distributions, i.e.\ the explicit distributions which occur on the right-hand side of the expressions in our limit theorems, neglecting only the errors of order $(\log n)^{-1}$.\\

The top row in Figure~\ref{simulations1} shows two cases covered by Theorem~\ref{Hallthm}. The top-left diagram is for $B = B(o,0.9)$, $A = B(o,1)$, which meets condition A3. We have $d=2$ and $k=1$, so there is no ``correction'' to the limiting distribution.
This is the only diagram in which we have taken $n$ larger than $10^4$. When plotted for $n = 10^4$ (not pictured), the distance between the empirical distribution and the limiting distribution appears to be smaller than for $n = 10^6$, but the shapes of the curves are very different, indicating that there is still a boundary effect influencing the distribution of the two-sample coverage threshold.

The top-right diagram is for our results when points are placed on the 2-dimensional unit torus. As a remark following Theorem~\ref{Hallthm} states, the proof of that theorem would generalise to this setting, giving exactly the same result. We have simulated $R_{n,m(n),k}$ for $k=3$, which is a case covered by our Theorem~\ref{Hallthm} but not included in the results of \cite{IY}.\\

Both diagrams on the bottom row of Figure~\ref{simulations1} are representations of the same simulation,
with points placed inside the two-dimensional unit disc, which certainly has a smooth boundary.
The inclusion of the explicit term of order $(\log n)^{-1/2}$
improves the accuracy of the estimated distribution considerably, as can be seen from the fact that the red dotted curve in the left diagram is much closer to the empirical distribution than the black dashed curve. In this $d=2$, $k=1$ setting the correction is of a larger order than the $O( \frac{\log\log n}{\log n} )$ terms in all of the other settings.

We remarked after stating Theorem~\ref{thsmoothgen} that $n \pi f_0 R_{n,m(n)}^2 - \mu(n \pi f_0 R_{n,m(n)}^2) \tod \Gum_{\log\log 2,1}$,
where $\mu(\cdot)$ is the median and $\Gum_{\log\log 2,1}$ is a Gumbel distribution with scale parameter 1 and median 0.
To illustrate this, in the second diagram on the second row of Figure~\ref{simulations1} we have translated all of the curves from the first diagram so that they pass through $(0,1/2)$,
i.e. so that they are the distributions of random variables with median $0$.
We can see that the corrected distribution is very close to the empirical distribution from the simulation, indicating that the \emph{shape} of the corrected limiting distribution closely matches the actual distribution of $n \pi f_0 R_{n,m(n)}^2$ for finite $n$, but with an offset corresponding to the difference between $\log n$ and the median of $n \pi f_0 R_{n,m(n)}^2$.

In the setting of Theorem~\ref{thsmoothgen},
the presence of a boundary has an effect on the distribution of $n \pi f_0 R_{n,m(n)}^2$ which disappears as $n \to \infty$, so is not reflected in the limit. Broadly speaking, the terms involving $e^{-\beta}$ come from the interior, and terms involving $e^{-\beta/2}$ come from the boundary. Our correction term corrects the shape of the distribution to account for these boundary effects.

The blue curve in the left-hand diagram was translated by the \emph{sample} median in order to pass through $(0,1/2)$ in the right-hand diagram. However, for applications of these limit theorems to real data, it is unlikely that tens of thousands of independent
samples of $R_{n,m}$ will be available to estimate the median of the distribution.\\

Theorem~\ref{thsmoothgen} covers two cases for $d=2$, $k=1$: when $A$ has a smooth boundary, and when $A$ is a polygon.
The first diagram in Figure~\ref{simulations2}
is in this latter case, with $A = [0,1]^2$.
If we compare this diagram with the bottom-left diagram of Figure~\ref{simulations1},
which is also for $d=2$, $k=1$ but with $A = B(o,1)$,
all of the same qualitative features can be observed:
a fairly large gap between the empirical distribution and the limit, a large improvement due to the correction, and an ``overshoot'' so the corrected distribution approximates the empirical distribution from the right-hand side while the limiting distribution is to the left.

This indicates that the behaviour of the two-sample coverage threshold (at least in two dimensions) is not strongly affected by the presence of ``corners'' on the boundary of $A$. It is likely that in higher dimensions, the limiting behaviour of $R_{n,m(n)}$ when $A$ is a polytope would be different from the behaviour when $A$ has a smooth boundary, as was observed for the \emph{coverage threshold} in \cite{Pen22}.

The top-right diagram in Figure~\ref{simulations2} is for $d=2$, $k=2$ with points inside the unit disc, which is the setting of the first limit result in Theorem~\ref{t:dhi}. The $d=2$, $k=2$ case is unique in that the limiting distribution has two terms, corresponding to the boundary and interior. In the other settings the limiting distribution for the position of the point in $\mathcal{Y}_j$ which is last to be $k$-covered as the discs expand is either distributed according to Lebesgue measure on $A$, or according to a distribution supported on $\partial A$. However, the existence of both terms in the limit in \eqref{e:lim22} indicates that for $d=k=2$, the ``hardest point to $k$-cover'' has a mixed distribution: the sum of a measure supported on the interior of $A$ with a measure supported on $\partial A$.\\

The bottom row of Figure~\ref{simulations2} contains the distributions from two simulations with $d=3$ and $k=1$ inside the unit ball.
In the left diagram we have taken $\tau = 1$, and the corrected limit approximates the empirical distribution well. In the right diagram  we have taken $\tau = 100$. The empirical distribution is extremely far from the limiting distribution, and the correction term has the wrong sign, so the corrected limit is an even worse approximation to the empirical distribution than the uncorrected limit is.

The fact that the empirical distribution is far to the \emph{left} of the limit (i.e.\ that $R_{n,m(n),k}^d$ is generally smaller than the limit would predict) when $\tau$ is large is rather surprising. If we consider $\tau_n \uparrow \infty$ sufficiently fast as $n \to \infty$, then $R_{n,m(n),k}$ should approximate the \emph{coverage threshold} considered in \cite{Pen22}.
As we remarked after the statement of Theorem~\ref{t:dhi}, the coverage threshold is generally much larger than our $R_{n,m(n),k}$. In the case $d=3$, $k=1$, the coefficient of $\log\log n$ in the weak law for the coverage threshold corresponding to Theorem~\ref{t:dhi} is larger, and so we might expect that if $\tau$ is large than the empirical distribution for $n \theta_d f_0 R_{n,m(n),k}^d - (2-2/d)\log n - (2k-4+2/d)\log\log n$ in Figure~\ref{simulations2} would be far to the right of the limiting distribution.

To explain the surprising fact that it is instead far to the left of the limit, we should examine Lemma~\ref{l:Poapprox}. In the ``Poissonized'' setting of that Lemma, given the configuration of ``transmitters'' $\mathcal{X}_t$, the conditional probability $\P[R_{t,\tau t,k}' \leq r_t | \mathcal{X}_t]$ is the probability that no point from $\mathcal{Y}_{\tau t}$ lies in the vacant region $V_{t,\tau t,k}$.
The lemma shows that when we replace the marginal probability $\P[R_{t,\tau t,k}' \leq r_t]$ with the probability that no point from $\mathcal{Y}_{\tau t}$ lies in a region of Lebesgue measure $\E V_{t,\tau t, k}$, the error induced is $O( (\log t)^{1-d} )$.
However, this is for fixed $\tau$. It can be seen from the proof that the error is proportional to $\tau^2 (\log t)^{1-d}$,
which is not negligible unless $t$ is very large compared to $\tau$.

To see why the corrected limiting cdf is \emph{below}
the empirical cdf, let $f(x) := e^{-\tau x}$.
In Lemma~\ref{l:Poapprox},
if $\Gamma_t := (t / |B|) |V_{t,r_t,k} \cap B |$,
then $\P[ R_{t,\tau t,k}'(B) \leq r_t ] = \E [f( \Gamma )]$,
while $e^{-\tau \gamma_t(B)} = f( \E[\Gamma] )$.
Hence by Jensen's inequality, $e^{-\tau \gamma_t}$
can only ever be an \emph{underestimate}
for $\P[ R_{t,\tau t,k}'(B) \leq r_t ]$,
with an error proportional to $\tau^2$.
All of our corrected expressions in Theorem~\ref{t:dhi}
are approximations of $e^{-\tau \gamma_t}$.

If we think of $\mathcal{X}_n$ as a set of transmitters and $\mathcal{Y}_{m(n)}$ as a set of receivers, then for most applications we would expect $\tau_n$ to be large. It should be possible to improve the estimate in this case by computing the leading-order error terms in Lemma~\ref{l:Poapprox}, using moments of $t|V_{t,\tau t,k}|$ or otherwise.\\

{\bf Acknowledgements.}
We thank Keith Briggs for suggesting this problem,
and for some useful discussions regarding simulations. \\

{\bf Funding Declaration.} All three authors
were supported in doing this research by 
%the Engineering and Physical Sciences Research Council
EPSRC grant EP/TO28653/1.\\

%{\bf Competing interest declaration.}
%The authors declare that the they have no competing interests as defined
%by Springer, or other interests that might be perceived to influence the
%results and/or discussion reported in this paper.  \\

{\bf Availability of Data and Materials declaration.}
The code for the simulations discussed in Section~\ref{sim-section} is available at 
\url{https://github.com/frankiehiggs/CovXY}
and the samples generated by that code are available at 
\url{https://researchdata.bath.ac.uk/id/eprint/1359}. \\

%{\bf Author contributions statement.}
%Penrose and Yang wrote most of Sections 4 and 5.
%Higgs did all the simulations and wrote the text describing
%these in Section 6. Penrose and Higgs wrote most of 
%Section 3.
%%as well as parts of Sections 1-3 including the geometrical lemmas. 
%All authors contributed to Sections 1 and 2,  and  reviewed the manuscript.

\end{document}